\documentclass[12pt,a4paper]{amsart}
\date{November 13, 2011}   
\usepackage{hyperref}
\usepackage{latexsym,amsmath,amsfonts,amscd,amssymb,graphics}
\usepackage[all]{xy}
\usepackage[ansinew]{inputenc}

\theoremstyle{plain}  
\newtheorem{theorem}{Theorem}[section]
\newtheorem*{theorem*}{Theorem}

\newtheorem{lemma}[theorem]{Lemma}
\newtheorem{proposition}[theorem]{Proposition}

\theoremstyle{definition}

\newtheorem*{notation*}{Notation}
\newtheorem{definition}[theorem]{Definition}
\newtheorem{assumption}[theorem]{Assumption}

\theoremstyle{remark}
\newtheorem{example}[theorem]{Example}
\newtheorem{remark}[theorem]{Remark}

\newtheorem*{claim*}{Claim}

\numberwithin{equation}{section}

\newcommand{\suchthat}{\;|\;}

\renewcommand{\leq}{\leqslant}

\renewcommand{\geq}{\geqslant}

\renewcommand{\setminus}{\smallsetminus}

\newcommand{\into}{\hookrightarrow}

\newcommand{\C}{\mathbb{C}}

\newcommand{\cH}{\mathcal{H}}
\newcommand{\cF}{\mathcal{F}}
\newcommand{\cM}{\mathcal{M}}

\newcommand{\xra}{\xrightarrow}

\DeclareMathOperator{\Jac}{Jac}

\DeclareMathOperator{\divisor}{div}
\DeclareMathOperator{\Prym}{Prym}
\DeclareMathOperator{\cNm}{Nm}

\DeclareMathOperator{\Supp}{Supp}
\DeclareMathOperator{\ord}{ord} 
 
\DeclareMathOperator{\tr}{tr}
\DeclareMathOperator{\rk}{rk}

\DeclareMathOperator{\End}{End}
 
\DeclareMathOperator{\Id}{Id}

\newcommand{\Pic}{\operatorname{Pic}}

\newcommand{\Dbar}{\overline{D}}

\let\oldmarginpar\marginpar
\renewcommand\marginpar[1]{\oldmarginpar{\tiny\bf\begin{flushleft} #1
\end{flushleft}}}

\begin{document}

\title[The singular fibre of the Hitchin map]{The singular fibre of the Hitchin map}

\author[P. B. Gothen]{Peter B. Gothen}
\address{Centro de Matem\'atica da Universidade do Porto\\
Faculdade de Ci\^encias, Universidade do Porto \\
Rua do Campo Alegre 687 \\ 4169-007 Porto \\ Portugal }
\email{pbgothen@fc.up.pt}

\author[A. G. Oliveira]{Andr\'e G. Oliveira}
\address{Departamento de Matem\'atica\\
  Universidade de Tr\'as-os-Montes e Alto Douro\\
  Quinta dos Prados, Apartado 1013 \\ 5000-911 Vila Real \\ Portugal }
\email{agoliv@utad.pt}

\thanks{
Members of VBAC (Vector Bundles on Algebraic Curves).
Partially supported by CRUP through Acção Integrada Luso-Espanhola no.\
E-38/09
and by the FCT (Portugal) with EU (COMPETE) and national funds 
through the projects PTDC/MAT/099275/2008 and PTDC/MAT/098770/2008,
and through Centro de Matemática da Universidade do Porto
(PEst-C/MAT/UI0144/2011, first author) and Centro de Matemática da
Universidade de Trás-os-Montes e Alto Douro (PEst-OE/MAT/UI4080/2011,
second author).
}


\begin{abstract}
  We give a description of the singular fibre of the Hitchin map on
  the moduli space of $L$-twisted Higgs pairs of rank 2 with fixed
  determinant bundle, when the corresponding spectral curve has any
  singularity of type $A_{m-1}$. In particular, we prove
  directly that this fibre is connected.
\end{abstract}

\maketitle


\section{Introduction}

Let $X$ be a compact Riemann surface of genus $g$ and let $L \to X$ be
a holomorphic line bundle. An \emph{$L$-twisted Higgs pair} or
\emph{Hitchin pair} is a pair $(E,\varphi)$, where $E \to X$ is a
holomorphic vector bundle and $\varphi \in H^0(X, \End(E) \otimes L)$
is an $L$-twisted endomorphism. In particular, if $K$ denotes the
canonical line bundle of $X$, then a $K$-twisted Higgs pair is a Higgs
bundle.

Denote by $\cM^\Lambda_L$ the moduli space of semistable $L$-twisted Higgs pairs
of rank $2$ with fixed determinant bundle $\Lambda \to X$ and
$\tr(\varphi)=0$. The Hitchin map
\begin{align*}
  \mathcal{H}\colon \cM^\Lambda_L &\longrightarrow H^0(X,L^2) \\
  (E,\varphi) &\longmapsto \det(\varphi)
\end{align*}
is a proper map. It plays a central role in many important aspects of
Higgs bundle theory. Two examples are integrable systems (see, e.g.,
Hitchin \cite{hitchin:1987b}, Bottacin \cite{bottacin:1995}, Markman
\cite{markman:1994} and Donagi--Markman \cite{donagi-markman:1996})
and the study of special representations of surface groups (now known
as Hitchin representations) initiated by Hitchin~\cite{hitchin:1992}.
More recently, the Hitchin map played a crucial role, for example, in
the work of Kapustin and Witten \cite{kapustin-witten:2007}, Frenkel
and Witten \cite{frenkel-witten:2007} and Ng\^o \cite{ngo:2006,ngo:2010}.

For a section $s \in H^0(X,L^2)$, let $X_s$ be the spectral curve
given by the zero locus of $s$ inside the total space of $L$. Then
$X_s \to X$ is a double cover. Note that if $s = \det(\varphi)$ then the
preimage of $x \in X$ corresponds to the eigenvalues of $\varphi_x\colon
E_x \to E_x \otimes L_x$. 

The spectral curve $X_s$ is smooth if and only if the divisor of zeros
of $s$ is reduced. It is an essential feature of the integrable
systems picture for the Hitchin map that its fibre over such a generic
$s$ is an abelian variety. In the present case this abelian variety is
the Prym variety of the double cover $X_s \to X$, i.e.\ the part of
the Jacobian of $X_s$ which is anti-invariant under the natural
involution induced by exchanging the sheets of the double cover.

At the other extreme, the most special fibre of the Hitchin map is the
fibre over zero, which is known as the nilpotent cone. The nilpotent
cone is singular and has a complicated structure: it contains the
fixed point locus of the natural $\C^*$-action on
$\mathcal{M}^{\Lambda}_L$ (given by multiplication on the Higgs field)
and has an irreducible component for each component of this fixed
point locus. In particular, the nilpotent cone contains the moduli
space of bundles as the locus $\varphi=0$. Since the nilpotent cone
encodes the topology of the moduli space, it has been studied
extensively (see e.g.\ \cite{biswas-gothen-logares:2009} and
references therein).

Our main goal in this paper is to give a description of the remaining
special fibres of the Hitchin map, corresponding to the case of $s
= \det(\varphi)$ being non-zero and having at least one multiple zero.
Here there are two cases to distinguish, corresponding to whether the
spectral curve is irreducible or not. 

When the spectral case is irreducible, we use the correspondence
between Higgs pairs on $X$ and rank one torsion free sheaves on $X_s$
to show that the fibre of the Hitchin map is essentially the
compactification by rank $1$ torsion free sheaves of the Prym of the
double cover $X_s \to X$ (see Theorem~\ref{fibre of Hitchin map} below
for the detailed statement). In order to prove this, we make use of
the the compactification of the Jacobian of $X_s$ by the parabolic
modules of Cook \cite{cook:1993,cook:1998} in order to describe the
fibre correctly. One advantage of this compactification is that it
fibers over the Jacobian of the normalization of $X_s$, as opposed to
the compactification by rank one torsion free sheaves.

In the case of reducible spectral curve, we resort to a direct
description of the fibre as a stratified space
(Theorem~\ref{prop:fibre when spectral curve is reducible}
below). This approach seems to us simpler than attempting to use the
spectral curve approach.

All together, our results allow us to prove the following main result
(Theorem~\ref{thm:main}). Again, the description of the fibre of the
Hitchin map via parabolic modules is an essential ingredient in the
proof.

\begin{theorem*}
  Assume that $\deg(L)>0$.  For any $s\in H^0(X,L^2)$, the fibre of
  the Hitchin map $\cH:\cM^\Lambda_L\to H^0(X,L^2)$ is
  connected. Moreover, if $s \neq 0$, the dimension of the fibre is
  $\dim(\cH^{-1}(s))=d_L+g-1$.
\end{theorem*}

We should point out that if $(2,d)=1$, the degree of $L$ is greater
than or equal to $2g-2$ and $L^2 \neq K^2$, then the moduli space is
known to be irreducible and then the connectedness of the generic
fibre can be used to prove connectedness of all fibres (see
Proposition~\ref{prop:big-L-connected-fibres} below).  Thus, it is
important to notice that our results apply independently of the degree
of the twisting line bundle $L$ (as opposed to most other studies of
the moduli space of twisted Higgs pairs carried out so far).

We also point out that Frenkel and Witten
\cite[Sec.~5.2.2]{frenkel-witten:2007} proved the connectedness of the
fibre in the particular case of irreducible spectral curve having only
simple nodes as singularities. Their argument makes implicit use of
the parabolic line bundles introduced by Bhosle in \cite{bhosle:1992}.
In the case of simple nodes parabolic line bundles are the same as
parabolic modules but, in general, they are different objects.

It seems quite likely that our relatively explicit description of the
fibre could be useful as a tool for further study of
geometric and topological properties of the singular fibre and we hope
to come back to this question on another occasion.

To finish this introduction, we give an outline of the organization of
the paper. In Sections~\ref{sec:higgs-pairs} and
\ref{sec:spectral-curve} we give some background on Higgs pairs and
the spectral curve and in
Section~\ref{sec:jacobian-irreducible-spectral-curve} we review the
theory of line bundles and the Prym variety in the case of singular
irreducible spectral curve. Then, in Section~\ref{Compactified jacobian
  and GPLB}, we introduce the parabolic modules and give some results,
which are used in Section~\ref{The singular case} to describe the
fibre of the Hitchin map in the case of irreducible spectral curve. In
Section~\ref{The non-integral case} we deal with the case of reducible
spectral curve. Finally in Section~\ref{thm:main} we put everything
together to prove our main theorem.

\subsection*{Acknowledgments}

We thank U. Bhosle, N. Hitchin and T. Pantev for useful
discussions. We also thank an anonymous referee for helpful comments
and in particular for pointing out that the connectedness result could
be extended to also cover the fibre of the Hitchin map over zero.

\section{Higgs pairs}
\label{sec:higgs-pairs}

Let $X$ be a compact Riemann surface of genus $g\geq 2$ and let $K =
K_X$ be the canonical bundle of $X$. Denote by
$\Jac(X) = \Pic^0(X)$ the Jacobian of $X$, which parametrizes degree
$0$ holomorphic line bundles on $X$; we also write $\Jac^d(X) =
\Pic^d(X)$.

\begin{definition}
  Let $L$ be a fixed holomorphic line bundle over $X$.  An
  \emph{$L$-twisted Higgs pair} of type $(n,d)$ over $X$ is a pair
  $(V,\varphi)$, where $V$ is a holomorphic vector bundle over $X$,
  with $\rk(V)=n$ and $\deg(V)=d$, and $\varphi$ is a global
  holomorphic section of $\End(V)\otimes L$, called the \emph{Higgs
    field}. A \emph{Higgs bundle} is a $K$-twisted Higgs pair.
\end{definition}

Throughout the paper we shall make the following assumption (see
Remark~\ref{rem:deg-L-zero} for comments on the case $\deg(L)=0$).

\begin{assumption}
  The line bundle $L \to X$ satisfies $\deg(L) >0$.
\end{assumption}

\begin{definition} Two $L$-twisted Higgs pairs $(V,\varphi)$ and
  $(V',\varphi')$ are \emph{isomorphic} if there is a holomorphic
  isomorphism $f:V\to V'$ such that $\varphi'f=(f\otimes 1_L)\varphi$.
\end{definition}

Using GIT, Nitsure \cite{nitsure:1991} constructed the moduli space
$\cM_L(n,d)$ of $S$-equivalence classes of rank $n$ and degree $d$
semistable $L$-twisted Higgs pairs over $X$. In this paper we shall
only be concerned with the case $n=2$.  In this case, the definition
of stability of $L$-twisted Higgs pairs takes the following form.

\begin{definition}
An $L$-twisted Higgs pair $(V,\varphi)$ of type $(2,d)$ is:
\begin{itemize}
\item \emph{stable} if $\deg(N)<d/2$ for any line bundle $N\subset V$
  such that $\varphi(N)\subset NL$.
 \item \emph{semistable} if $\deg(N)\leq d/2$ for any line bundle
   $N\subset V$ such that $\varphi(N)\subset NL$.
 \item \emph{polystable} if is semistable and for any line bundle
   $N_1\subset V$ such that $\varphi(N_1)\subset N_1L$ and $\deg(N_1)
   = d/2$, there is another line bundle $N_2\subset V$ such that
   $\varphi(N_2)\subset N_2L$ and $V = N_1 \oplus N_2$.
\end{itemize}
\end{definition}

The notion of $S$-equivalence for Higgs pairs is defined analogously
to the case of vector bundles and, as in that case, each
$S$-equivalence class contains a unique polystable representative.

The following is known from
\cite[Theorem~1.2]{biswas-gothen-logares:2009} (see also
Nitsure \cite[Proposition~7.4]{nitsure:1991} and Simpson
\cite[Theorem~11.1]{simpson:1995}): if $(n,d)=1$, $\deg(L) \geq 2g-2$, and either
$L = K$ or $L^n \neq K^n$, then the moduli space $\cM_L(n,d)$ is
smooth (and irreducible) with
\begin{displaymath}
  \dim \cM_L(n,d) = n^2\deg(L) + 1+ \dim H^1(X,L).
\end{displaymath}

There is a map 
\begin{align*}
  p:\cM_L(n,d)&\longrightarrow\Jac^d(X)\times H^0(X,L)\\
  (V,\varphi)&\longmapsto(\Lambda^nV,\tr(\varphi)).
\end{align*}
For each fixed $\Lambda \in \Jac^d(X)$, we define the
\emph{moduli space of $L$-twisted Higgs pairs of type $(n,d)$ with
  fixed determinant $\Lambda$} to be the subvariety of $\cM_L(n,d)$
given by the fibre of $p$ over $(\Lambda,0)$:
\begin{displaymath}
\cM_L^\Lambda(n,d)=p^{-1}(\Lambda,0).
\end{displaymath}
Accordingly, we say that $(V,\varphi)$ is an \emph{$L$-twisted Higgs pair with fixed
determinant $\Lambda$}, if $\Lambda^2V \cong \Lambda$ and
$\tr(\varphi)=0$.

Nitsure \cite[Theorem~7.5]{nitsure:1991} proved that in rank two, the
moduli space is connected for any $L$. His proof goes through
unchanged in the fixed determinant case, so we have the following
result.

\begin{proposition}
  \label{prop:moduli-connected}
  For any $L$, $d$ and $\Lambda$, the moduli spaces $\cM_L(2,d)$ and $\cM_L^\Lambda(2,d)$
  are connected. \qed
\end{proposition}

Henceforth, we shall assume that $n=2$ and also drop $n$ and $d$ from the
notation, writing simply $\cM_L$ and $\cM_L^\Lambda$ for the rank two
moduli spaces. 

The Hitchin map is defined by taking the
characteristic polynomial of $\varphi$. Thus, on the rank two, fixed
determinant moduli space, it is given as follows:

\begin{definition}
  The \emph{Hitchin map} on $\cM_L^\Lambda$ is the map
$$\begin{array}{rccl}
  \cH:&\cM_L^\Lambda & \longrightarrow & H^0(X,L^2)\\
  &(V,\varphi) & \longmapsto & \det(\varphi).
\end{array}$$ 
\end{definition}

\section{The spectral curve}
\label{sec:spectral-curve}

In this section we recall the construction and basic properties of the
spectral curve; see
\cite{beauville-narasimhan-ramanan:1989,hitchin:1987,hitchin:1987b}
for details. The following notation will be used for the
remainder of the paper: we denote the degree of the line bundle $L$ by
$$d_L=\deg(L)>0$$
and write
$$D_s=\divisor(s)\in\mathrm{Sym}^{2d_L}(X)$$
for the divisor of a section $s \in H^0(X,L^2)$.

\subsection{The spectral curve and Higgs pairs}

We begin by reviewing the construction of the spectral curve $X_s$
associated to a section $s\in H^0(X,L^2)$. Consider the complex
surface $T$ given by the total space of the line bundle $L$, and let
$\pi:T\to X$ be the projection. The pullback $\pi^*L$ of $L$ to its
total space has a tautological section
$$\lambda\in H^0(T,\pi^*L)$$ defined by $\lambda(x)=x$.

\begin{definition}
  Let $s\in H^0(X,L^2)$. The \emph{spectral curve} $X_s$ associated to
  $s$ is the zero scheme in the surface $T$ of the section
  $$\lambda^2+\pi^*s\in H^0(T,\pi^*L^2).$$
\end{definition}

The following simple observations (cf. Section 3 of
\cite{beauville-narasimhan-ramanan:1989}) will be of relevance later.

\begin{remark}
  \label{rem:singular-smooth-X_s}
  The spectral curve $X_s$ is always reduced, but it may be singular
  and reducible. In fact, it is smooth if and only if $s$ only has
  simple zeros and it is irreducible if and only if $s$ is not the
  square of a section of $L$.
  Note that if $D_s=2\widetilde D$ for some divisor $\widetilde D$,
  then $L\cong\mathcal O(\widetilde D)\otimes N$ where $N$ is a
  $2$-torsion point of the Jacobian and, in this situation, $X_s$ is
  reducible if and only if $N = \mathcal{O}$.
\end{remark}

In the present setting, the fundamental result on the relation between
the spectral curve and Higgs pairs can be formulated as follows
(cf.~\cite[Proposition 3.6]{beauville-narasimhan-ramanan:1989} and \cite{hitchin:1987b}):

\begin{theorem}
  \label{BNR}
  Let $s\in H^0(X,L^2)$ be such that the spectral curve $X_s$ is
  irreducible. Then there is a bijective correspondence between
  isomorphism classes of torsion-free sheaves of rank $1$ on $X_s$ and
  isomorphism classes of $L$-twisted Higgs pairs $(V,\varphi)$ of rank
  $2$, where $\varphi:V\to V\otimes L$ is a homomorphism with
  $\det(\varphi)=s$ and $\tr(\varphi)=0$.  The correspondence is given
  by associating to such a sheaf $\mathcal F$ on $X_s$, the sheaf
  $\pi_*\mathcal F$ on $X$ and the homomorphism $\pi_*\mathcal
  F\to\pi_*\mathcal F\otimes L\cong\pi_*(\mathcal F\otimes\pi^*L)$
  given by multiplication by the canonical section $\lambda\in
  H^0(X_s,\pi^*L)$.
\end{theorem}

This theorem will be the main tool to describe the fibre of $\cH$ over
a non-zero section $s\in H^0(X,L^2)$.

\subsection{The generic fibre of the Hitchin map}

The material of this subsection is well-known but we include it for
completeness. Let $s\in H^0(X,L^2)$ be a section with simple
zeros. Then $X_s$ is smooth and we have a double cover 
\begin{equation}\label{eq: spectral curve double cover of X}
\pi:X_s\longrightarrow X,
\end{equation}
of $X$, ramified over $D_s$.
Since $X_s$ is smooth, the torsion free sheaf $\mathcal F$ on $X_s$,
which corresponds to $(V,\varphi)$ under Theorem~\ref{BNR}, is in fact
a line bundle $F$.

By the Riemann--Hurwitz formula, the genus of $X_s$ is
\begin{equation}\label{genus spectral curve}
g(X_s)=2g-1+d_L.
\end{equation}
Moreover, since $\pi$ has discrete fibres, we have $H^i(X_s,\cF)\cong
H^i(X,\pi_*\cF)$, $i=0,1$, for any coherent sheaf $\cF$ on
$X_s$. Hence, from the Riemann-Roch theorem and from 
(\ref{genus spectral curve}), 
it follows that $$\deg(\pi_*\cF)=\deg(\cF)-d_L.$$ 

Let $\mathcal T_{\pi}\subset \Jac^{d+d_L}(X_s)$ be the space of line
bundles on $X_s$ such that the determinant of their push forward under
$\pi$ is $\Lambda\in\Jac^d(X)$:
\begin{equation}\label{definition of Taupi}
\mathcal T_{\pi}=\{F\in\Jac^{d+d_L}(X_s)\suchthat\det(\pi_*F)\cong \Lambda\}. 
\end{equation}

We will now recall the definition of Prym variety. Consider a connected double cover 
\begin{equation}\label{eq:double cover of X}
p:X'\longrightarrow X
\end{equation} of $X$. Given a divisor $E=\sum n_iq_i$ in $X'$, the norm $\cNm_p(E)$ of $E$ is the divisor in $X$ defined by $\sum
n_ip(q_i)$. In terms of line bundles, this gives rise to the norm
map $$\cNm_p:\Jac(X')\longrightarrow\Jac(X)$$ which is a group
homomorphism.

\begin{definition}\label{def:Prym}
  The \emph{Prym variety of $X'$} 
  $$\Prym_p(X')=\{N\in\Jac(X')\suchthat
  \cNm_p(N)\cong\mathcal O_{X'}\}$$ is the kernel of the norm map
  with respect to $p$.
\end{definition}

\begin{remark}
\label{rem:prym-components}
 It is proved by Mumford in \cite{mumford:1971} that if $p:X'\to X$ in (\ref{eq:double cover of X}) is a ramified double cover, then $\Prym_p(X')$ is connected, whereas if $p$ is a non-trivial unramified cover, then $\Prym_p(X')$ has two connected components.
In the latter case, the Prym variety of $X'$ is, sometimes, defined just as the connected
  component of the kernel of $\cNm_p$ which contains the
  identity. It is, however, important to note that we define the Prym to be the
  full kernel of $\cNm_p$.
\end{remark}

We will be interested in $\Prym_\pi(X_s)$, where $\pi$ is the double cover (\ref{eq: spectral curve double cover of X}). For any line bundle $F$ in $X_s$, it is known
(see e.g.\ \cite{beauville-narasimhan-ramanan:1989}) that
\begin{equation}\label{det and norm}
\det(\pi_*F)\cong\cNm_{\pi}(F)\otimes L^{-1}
\end{equation}
so we can rewrite $\mathcal T_{\pi}$ defined in 
(\ref{definition of Taupi}) as
$$\mathcal T_{\pi}
=\{F\in\Jac^{d+d_L}(X_s)\suchthat\cNm_{\pi}(F)\cong\Lambda L\}.$$ For
each choice of a fixed element $F_0$ of $\mathcal T_{\pi}$, we
therefore obtain a
non-canonical isomorphism
\begin{equation}\label{isom T cong Prym}
  \begin{aligned}
    \mathcal T_{\pi}&\cong\Prym_{\pi}(X_s)\\
    F &\mapsto FF_0^{-1}.
  \end{aligned}
\end{equation}

\begin{proposition}\label{fibre smooth case = Prym}
  Let $s\in H^0(X,L^2)$ have simple zeros. Then there is an
  isomorphism between $\cH^{-1}(s)$ and $\mathcal T_{\pi}$.
\end{proposition}

\proof 
Theorem~\ref{BNR} gives a bijection between $\mathcal T_{\pi}$ and
isomorphism classes of pairs Higgs pairs $(V,\varphi)$ with fixed
determinant $\Lambda$.
It remains to see that any $(V,\varphi)$ thus obtained is stable.
Suppose that there is a line subbundle $N\subset V=\pi_*F$ such that
$\varphi(N)\subset NL$. Then $N$ is an eigenbundle of $\varphi$. If
the corresponding eigenvalue is $\lambda\in H^0(X,L)$, then the other
eigenvalue is $-\lambda$ and hence $\det(\varphi)=-\lambda^2$. It
follows that $X_s$ is not smooth (in fact, not even irreducible), so
there is no such $N$. The Higgs pair $(V,\varphi)$ is therefore
stable.
\endproof

The Prym variety is a principally polarized abelian variety of
dimension $g(X_s) - g = d_L+g-1$. Hence it follows from
Proposition~\ref{fibre smooth case = Prym} and 
(\ref{isom T cong Prym}) that the dimension of the
generic fibre of $\cH$ over $s$ is
\begin{equation}\label{dimension of generic fibre}
\dim \mathcal{H}^{-1}(s) = d_L+g-1.
\end{equation}
In this generic case $s$ has simple zeros and the double cover (\ref{eq: spectral curve double cover of X}) is ramified over $D_s$, thus it follows from (\ref{isom T cong Prym}) that $\mathcal{H}^{-1}(s)$ is connected.

Under certain conditions on the degree of $L$, the connectedness of
the generic fibre of the Hitchin map actually implies that all fibres
are connected. To be precise, we have the following result.

\begin{proposition}
  \label{prop:big-L-connected-fibres}
  Assume that $(2,d)=1$ and that moreover $\deg(L) \geq 2g-2$ and $L^2
  \neq K^2$. Then the fibres of the Hitchin map are connected.
\end{proposition}

\begin{proof}
  By Stein factorization
  (\cite[Corollary~III.11.15]{hartshorne:1977}), the Hitchin map
  factors as
  \begin{displaymath}
    \mathcal{H} = g \circ \mathcal{H}'\colon
    \mathcal{M}^{\Lambda}_L \xra{\mathcal{H}'} Y \xra{g} H^0(X,L^2),
  \end{displaymath}
  where $\mathcal{H}'$ has connected fibres and $g$ is a finite
  morphism.  Since the generic fibre of $\mathcal{H}$ is connected, so
  is the generic fibre of $g$. The hypotheses of the proposition
  guarantee that the moduli space is irreducible by
  \cite[Theorem~1.2]{biswas-gothen-logares:2009}. Hence
  $Y=\mathcal{H}'(\mathcal{M}^{\Lambda}_L)$ is irreducible and we can
  apply Zariski's main Theorem
  (\cite[Corollary~III.11.14]{hartshorne:1977}) and deduce that all
  fibres of $g$ are connected. The result follows.
\end{proof}

\begin{remark}
  \label{rem:higher-rank-big-L-connected}
  The previous proposition holds more generally for the moduli spaces
  $\mathcal{M}_L(n,d)$ and $\mathcal{M}^{\Lambda}_L(n,d)$, under the
  conditions $(n,d)=1$, $\deg(L) \geq 2g-2$ and $L^n \neq K^n$. This
  is because by the independent work of Bottacin (Theorems~4.7.2 and
  4.8.4 of \cite{bottacin:1995}) and Markman (Theorem~8.5 and
  Remark~8.7 of \cite{markman:1994}) the generic fibre of the
  Hitchin map is connected, and the moduli spaces are irreducible by
  \cite[Theorem~1.2]{biswas-gothen-logares:2009}.
\end{remark}


\section{Line bundles on an irreducible singular spectral curve}
\label{sec:jacobian-irreducible-spectral-curve}

We now move on to give a description of the fibre of the Hitchin map
$\mathcal{H}$ over $s\in H^0(X,L^2)$ when $X_s$ is singular. Thus we
consider the situation when the section $s$ has multiple zeros (cf.\
Remark~\ref{rem:singular-smooth-X_s}). There are two different cases
to be considered:
\begin{enumerate}
\item The singular spectral curve $X_s$ is irreducible. This case will
  be studied in the present section, Section~\ref{Compactified
    jacobian and GPLB} and Section~\ref{The singular case}.
\item The spectral curve $X_s$ is singular and has two irreducible
  components. This case will be studied in 
  Section~\ref{The non-integral case}.
\end{enumerate}
Recall that we write $D_s=\divisor(s)$. Note that case (2) occurs exactly when $D_s=2\widetilde D$ and
$\mathcal O(\widetilde D)\cong L$. Thus, in the remainder of this
section and in Sections~\ref{Compactified jacobian and GPLB} and
\ref{The singular case} we shall consider the case when $s\in
H^0(X,L^2)$ has a multiple zero, and moreover assume that if
$D_s=2\widetilde D$ for some $\widetilde{D}$, then the line bundle $L$
is not isomorphic to $\mathcal O(\widetilde D)$. The spectral curve
$X_s$ is hence singular and irreducible.

\subsection{The Jacobian}
\label{The spectral curve Xs for s with multiple zeros}

Suppose then that $q\in X$ is a zero of $s$ with multiplicity $m$.
Consider a local coordinate $z$ on $X$ centered in $q$ such that $X_s$
is defined locally by the equation $$x^2-z^m=0.$$ Let
$p=\pi^{-1}(q)\in X_s$. If $m>1$, $p$ is a singular double point of
$X_s$; it is a simple singularity of type $A_{m-1}$ (see for instance
Chapter 2, Sec. 8 of \cite{barth-hulek-peters-van de ven:2004}).

The normalization
$$\widetilde\pi:\widetilde X_s\longrightarrow X_s$$
is then a smooth curve and
$\widetilde\pi$ is an isomorphism outside of $\widetilde\pi^{-1}(X_s^{\text{sing}})$, where $X_s^{\text{sing}}$ denotes the singular locus of $X_s$.  
If $m$ is even, then $p$ is a
node (ordinary, if $m=2$) and $\widetilde\pi^{-1}(p)=\{p_1,p_2\}$ with
$p_1\neq p_2$. If $m\geq 3$ is odd, $p$ is a cusp and
$\widetilde\pi^{-1}(p)=p_1$.

The following result is well-known. We include the proof for
the convenience of the reader.
\begin{proposition}\label{jacobian singular}
  Suppose that $X_s$ has $r_1$ nodes of types $A_{m_i-1}$,
  $i=1,\dots,r_1$ ($m_i$ even) and, that it has $r_2$ cusps of types
  $A_{m'_j-1}$, $j=1,\dots,r_2$ ($m'_j$ odd). Then there is a short exact
  sequence
  \begin{displaymath}
    0\longrightarrow (\C^*)^{r_1}\times\C^{\sum_{i=1}^{r_1}(m_i-2)/2+\sum_{j=1}^{r_2} (m'_j-1)/2}\longrightarrow\Jac(X_s)\longrightarrow\Jac(\widetilde X_s)\longrightarrow 0.
  \end{displaymath}
\end{proposition}

\proof 

Assume first that the curve $X_s$ has only one singularity. We
consider the cases of a node and of a cusp separately.

Let $p$ be a singular point of $X_s$ and $p \in U$ for an open $U
\subset X_s$, with
local equation $x^2-z^m=0$, where $m\geq 2$ is even. Around $p$, $X_s$
is reducible and we write $$U=U_1\cup U_2$$ for the decomposition of
$U$ into the two irreducible components. The component $U_1$
(resp.\ $U_2$) has then defining equation $x-z^{m/2}=0$
(resp.\ $x+z^{m/2}=0$). The coordinate ring of $U$
is 
$$\C[x,z]/(x^2-z^m).$$ 
The
desingularization of $U$ is given by two disjoint copies of $\C$ which
we denote by $\C_1$ and $\C_2$, and the normalization map
is $$f=f_1\cup f_2:\C_1\cup\C_2\longrightarrow U_1\cup U_2$$ defined
by
$$f(w_1)=f_1(w_1)=(w_1^{m/2},w_1),\ w_1\in\C_1\hspace{0,2cm}\text{ and }\hspace{0,2cm}f(w_2)=f_2(w_2)=(-w_2^{m/2},w_2),\ w_2\in\C_2.$$
The corresponding
map $\phi=\phi_1\oplus\phi_2$ between the coordinate rings is
\begin{equation}
\label{map between coordinate rings near singularity of even type}
\begin{aligned}
  \phi:\C[x,z]/(x^2-z^m) &\longrightarrow \C_1[t_1]\oplus\C_2[t_2], \\
  g(x,z) &\longmapsto g(t_1^{m/2},t_1)\oplus g(-t_2^{m/2},t_2).
\end{aligned}
\end{equation}
If we write
\begin{equation}\label{g(x,z)}
  g(x,z)=\sum a_{ij}x^iz^j 
\end{equation}
then 
$$
\phi_1(g(x,z))(t_1)=\sum a_{ij}t_1^{\frac{m}{2}i+j}
\qquad
\text{and} 
\qquad
\phi_2(g(x,z))(t_2)=\sum (-1)^ia_{ij}t_2^{\frac{m}{2}i+j}.
$$ 
As a result, if $\phi_i(g(x,z))^{(k)}(0)$ denotes the $k$-th derivative
of $\phi_i(g(x,z))$ at $0$, we obtain
$$\phi_1(g(x,z))^{(k)}(0)=\phi_2(g(x,z))^{(k)}(0)\in\C$$
for every $k=0,\dots,(m-2)/2$, and conclude that
\begin{equation}\label{local ring for even singularities}
  U\cong\text{Spec}\{(f_1,f_2)\in\C_1[t_1]\oplus\C_2[t_2]
  \suchthat f_1^{(k)}(0)=f_2^{(k)}(0),\;k=0,\ldots,(m-2)/2\}.
\end{equation}
Consider now only functions $g(x,z)$ as in (\ref{g(x,z)}) which do not
vanish at $p$, i.e., with $a_{00}\neq
0$. Then $$\phi_1(g(x,z))(0)=\phi_2(g(x,z))(0)\in\C^*$$ and also, if
$m\geq 4$, then
$$\phi_1(g(x,z))^{(k)}(0)=\phi_2(g(x,z))^{(k)}(0)\in\C$$ for every $k=0,\dots,(m-2)/2$. Let $(\C^*\times\C^{(m-2)/2})_p$ be the skyscraper sheaf supported at $p$ and consider the evaluation map $\widetilde\pi_*\mathcal O_{\widetilde X_s}^*\to (\C^*\times\C^{(m-2)/2})_p$ at $p$ given, for each open $U$, by
$$f\mapsto 0$$ if $p\notin U$, and
$$(f_1,f_2)\mapsto(f_1(p_1)/f_2(p_2),f_1'(p_1)-f_2'(p_2),\ldots,f_1^{((m-2)/2)}(p_1)-f_2^{((m-2)/2)}(p_2))$$
if $p\in U$. This has kernel the image of the sheaf map $\phi:\mathcal
O_{X_s}^*\to\widetilde\pi_*\mathcal O_{\widetilde X_s}^*$ given by
$\phi(U)=\Id$ if $p\notin U$ and by 
(\ref{map between coordinate rings near singularity of even type}) 
if $p\in U$. Hence we have the exact
sequence $$0\longrightarrow \mathcal O_{X_s}^*\stackrel{\phi}{\longrightarrow}\widetilde\pi_*\mathcal O_{\widetilde X_s}^*\longrightarrow (\C^*\times\C^{(m-2)/2})_p\longrightarrow 0.$$

We now consider cusps in the spectral curve. If $p$ is a singular point of $X_s$ and $U$ the open set of $X_s$ containing $p$ with local equation $x^2-z^m=0$, with $m\geq 3$ odd, then $U$ is now irreducible and one sees analogously to the previous case that there is a map  $\phi:\C[x,z]/(x^2-z^m)\to\C[t]$ between the corresponding coordinate rings of $p$ and of $p_1=\widetilde\pi^{-1}(p)$. $\phi$ is given by
\begin{equation}\label{map between coordinate rings near singularity of odd type}
 \phi(g(x,z))=g(t^m,t^2)
\end{equation}
and
$$U\cong\text{Spec}\{f\in\C[t]\suchthat f'(0)=f'''(0)=\dots=f^{(m-2)}(0)=0\}.$$

Consider the skyscraper sheaf $(\C^{(m-1)/2})_p$ supported at $p$ and again the evaluation $\widetilde\pi_*\mathcal O_{\widetilde X_s}^*\to (\C^{(m-1)/2})_p$ at $p$ given, for each open $U$, by
$$f\mapsto 0$$ if $p\notin U$, and
$$f\mapsto(f'(p),f'''(p),\dots,f^{(m-2)}(p))$$
if $p\in U$. This has kernel the image of the sheaf map $\phi:\mathcal O_{X_s}^*\to\widetilde\pi_*\mathcal O_{\widetilde X_s}^*$, given by $\phi(U)=\Id$ if $p\notin U$ and by (\ref{map between coordinate rings near singularity of odd type}) if $p\in U$. We have thus the following short exact sequence $$0\longrightarrow \mathcal O_{X_s}^*\stackrel{\phi}{\longrightarrow}\widetilde\pi_*\mathcal O_{\widetilde X_s}^*\longrightarrow (\C^{(m-1)/2})_p\longrightarrow 0.$$

In general, if $X_s$ has more singularities, each of which of type $A_{m-1}$, we have the short exact sequence of sheaves on $X_s$
$$0\longrightarrow \mathcal O_{X_s}^*\longrightarrow \widetilde\pi_*\mathcal O_{\widetilde X_s}^*\longrightarrow T\longrightarrow 0$$
where $T$ is a skyscraper sheaf supported at $X_s^\text{sing}$ such that, over each singular point, it is given as one of the above types.
Taking the corresponding cohomology sequence, we obtain
\begin{equation}\label{cohomology sequence Jacobian}
0\longrightarrow\bigoplus_{p\in X_s^\text{sing}}T_p\stackrel{\delta}{\longrightarrow}H^1(X_s,\mathcal O_{X_s}^*)\stackrel{\widetilde\pi^*}{\longrightarrow}H^1(X_s,\widetilde\pi_*\mathcal O_{\widetilde X_s}^*)\longrightarrow 0.
\end{equation}
Since $\widetilde\pi:\widetilde X_s\to X_s$ is a finite morphism, we have an isomorphism $$H^1(X_s,\widetilde\pi_*\mathcal O_{\widetilde X_s}^*)\cong H^1(\widetilde X_s,\mathcal O_{\widetilde X_s}^*)$$
hence, restricting sequence (\ref{cohomology sequence Jacobian}) to degree zero line bundles, we obtain
\begin{equation}\label{jacobian singular 1}
0\longrightarrow \bigoplus_{p\in X_s^\text{sing}}T_p\stackrel{\delta}{\longrightarrow}\Jac(X_s)\stackrel{\widetilde\pi^*}{\longrightarrow}\Jac(\widetilde X_s)\longrightarrow 0. 
\end{equation}
Since we are assuming that $X_s$ has $r_1$ nodes with types $A_{m_i-1}$, $i=1,\dots,r_1$ and, that it has $r_2$ cusps of types $A_{m'_j-1}$, $j=1,\dots,r_2$, then sequence (\ref{jacobian singular 1}) becomes
$$0\longrightarrow (\C^*)^{r_1}\times\C^{\sum_{i=1}^{r_1}(m_i-2)/2+\sum_{j=1}^{r_2} (m'_j-1)/2}\stackrel{\delta}{\longrightarrow}\Jac(X_s)\stackrel{\widetilde\pi^*}{\longrightarrow}\Jac(\widetilde X_s)\longrightarrow 0$$
and this finishes the proof.\endproof

\subsection{The Prym}
\label{sec:prym-normalization}
Consider the double cover, $$\overline\pi=\pi\circ\widetilde\pi:\widetilde X_s\longrightarrow X$$ of $X$:
$$\xymatrix{&\widetilde X_s\ar[dr]^{\widetilde\pi}\ar[dd]_{\overline\pi}&&\\
          &&X_s\ar[dl]^{\pi}&\\
	  & X.&&}$$
We have the
corresponding Prym variety (see Definition \ref{def:Prym}), $\Prym_{\bar{\pi}}(\tilde{X}_s) \subset
\Jac(\tilde{X}_s)$. The purpose of this section is to identify the
short exact sequence obtained from the one of
Proposition~\ref{jacobian singular} by restricting to this Prym.  We
continue under the hypothesis of that proposition, assuming that $X_s$
has $r_1$ nodes with types $A_{m_i-1}$, $i=1,\dots,r_1$, and $r_2$
cusps of types $A_{m'_j-1}$, $j=1,\dots,r_2$.

Let
\begin{equation}\label{Ds}
D_s=\sum_{i=1}^{r_1}m_iq_i+\sum_{j=1}^{r_2}m'_jq_j
\end{equation}
be the decomposition of the divisor of $s\in H^0(X,L^2)$ into its even and odd parts.
Notice that $\overline\pi:\widetilde X_s\to X$ is ramified exactly over the odd part of the divisor $D_s$.

\begin{remark}\label{r2 even}
Since $\deg(D_s)=2d_L$ is even, then $r_2$ must be even.
\end{remark}

Define the divisor $D'_s$ as
\begin{equation}\label{D's}
D'_s=\sum_{i=1}^{r_1}\frac{m_i}{2}q_i+\sum_{j=1}^{r_2}\frac{m'_j-1}{2}q_j.
\end{equation}
Then we have the short exact sequence
\begin{equation}
\label{eq:ses-Ds}
0\longrightarrow\mathcal
O_{X_s}\longrightarrow\widetilde\pi_*\mathcal O_{\widetilde
  X_s}\longrightarrow\mathcal O_{\pi^{-1}(D'_s)}\longrightarrow 0
\end{equation} 
where $D'_s$ is defined in (\ref{D's}) and $\pi^{-1}(D'_s)$ is the
obvious divisor in $X_s$.

\begin{remark}
  It follows from Proposition \ref{jacobian singular}
  that $$\dim\Jac(X_s)=\dim\Jac(\widetilde X_s)+\deg(D'_s)$$ hence,
  denoting by $g(X_s)$ the arithmetic genus of $X_s$ and $g(\widetilde
  X_s)$ the genus of $\widetilde X_s$,
$$g(X_s)=g(\widetilde X_s)+\deg(D'_s).$$
\end{remark}

\begin{lemma}\label{pullback from Prym to normalization}
Let $F$ be a line bundle over $X_s$. Then $\det(\bar\pi_*\widetilde\pi^*F)\cong\det(\pi_*F)\otimes\mathcal O_X(D'_s)$.
\end{lemma}
\proof
Tensoring the short exact sequence (\ref{eq:ses-Ds}) by $F$, noticing
that $F\otimes\widetilde\pi_*\mathcal O_{\widetilde
  X_s}\cong\widetilde\pi_*\widetilde\pi^*F$, and then applying the
push forward by $\pi$ (which has discrete fibres), we obtain the short
exact sequence $$0\longrightarrow \pi_*F\longrightarrow \bar\pi_*\widetilde\pi^*F\longrightarrow\mathcal O_{D'_s}\longrightarrow 0$$ on $X$.
The result follows by taking determinants.
\endproof

In analogy with (\ref{definition of Taupi}), we define
\begin{equation}\label{T_pi}
\mathcal
T_\pi=\{F\in\Jac^{d+d_L}(X_s)\suchthat\det(\pi_*F)\cong\Lambda\}
\subset \Jac^{d+d_L}(X_s).
\end{equation}
From Lemma~\ref{pullback from Prym to normalization},
the restriction to $\mathcal T_\pi$ of the map
$\widetilde\pi^*:\Jac^{d+d_L}(X_s)\to\Jac^{d+d_L}(\widetilde X_s)$ of 
Proposition~\ref{jacobian singular} (which holds for any degree) takes values
in 
$$\widetilde{\mathcal T}_{\overline\pi}
=\{F\in\Jac^{d+d_L}(\widetilde X_s)\suchthat
\det(\overline\pi_*F)\cong\Lambda(D'_s)\}.$$ If $F$ and $F'$ are two
line bundles over $X_s$ such that 
$\widetilde\pi^*F\cong\widetilde\pi^* F'$, then
$\det(\overline\pi_*\widetilde\pi^*F)=\det(\overline\pi_*\widetilde\pi^*F')$,
so, again from Lemma \ref{pullback from Prym to normalization},
$\det(\pi_*F)=\det(\pi_*F')$. This shows that the fibre of the
restriction of $\widetilde\pi^*$ to $\mathcal T_\pi$ is the same as
the fibre of $\widetilde\pi^*:\Jac^{d+d_L}(X_s)\to\Jac^{d+d_L}(\widetilde X_s)$ in
Proposition \ref{jacobian singular}. Hence, we have the sequence
$$0\longrightarrow (\C^*)^{r_1}\times\C^{\sum_{i=1}^{r_1}(m_i-2)/2+\sum_{j=1}^{r_2} (m'_j-1)/2}\longrightarrow\mathcal T_\pi\stackrel{\widetilde\pi^*}{\longrightarrow}\widetilde{\mathcal T}_{\overline\pi}\longrightarrow 0.$$

Recall that the spectral curve $X_s$ depends on the line bundle $L$ and on the section $s\in H^0(X,\mathcal{O}(D_s))$ with $L^2\cong\mathcal O(D_s)$. Conversely, if we have $D_s$ and a double cover $\pi:X_s\to X$ ramified over $D_s$, we recover $L$ since $\pi_*\mathcal O_{X_s}\cong\mathcal O_X\oplus L^{-1}$ (cf. Lemma 17.2 of \cite{barth-hulek-peters-van de ven:2004}).
In the same way, we recover the line bundle $\widetilde L$ over $X$, associated to the double cover $\overline\pi:\widetilde X_s\to X$.

Using the short exact sequence (\ref{eq:ses-Ds}) and the definitions
of $L$ and $\widetilde L$, we obtain the following diagram with exact
rows and column:
$$\xymatrix{&&&0\ar[d]&&\\
          &0\ar[r]&\mathcal O_X\ar[r]&\pi_*\mathcal O_{X_s}\ar[r]\ar[d]&L^{-1}\ar[r]&0\\
          &0\ar[r]&\mathcal O_X\ar[r]&\overline\pi_*\mathcal O_{\widetilde X_s}\ar[r]\ar[d]&\widetilde L^{-1}\ar[r]&0\\
          &&&\mathcal O_{D'_s}\ar[d]&&\\
	  &&&0.&&}$$
{}From this it follows that 
\begin{equation}\label{widetilde L in terms of L}
 \widetilde L\cong L(-D'_s).
\end{equation}

Now, the space $\mathcal T_\pi$ is clearly isomorphic to $P_\pi
\subset \Jac(X_s)$ defined by
\begin{equation}
  \label{eq:def-P_pi}
  P_\pi=\{F\in\Jac(X_s)\suchthat\det(\pi_*F)\cong L^{-1}\}
\end{equation}
and, from (\ref{widetilde L in terms of L}) and Lemma \ref{pullback from Prym to normalization}, 
$$\widetilde\pi^*(P_\pi)=\{F\in\Jac(\widetilde
X_s)\suchthat\det(\overline\pi_*F)\cong\widetilde L^{-1}\}.$$ 
But, since
$\det(\overline\pi_*N)\cong\cNm_{\overline\pi}(N)\otimes\widetilde
L^{-1}$ (cf. (\ref{det and norm})), we have
$$\widetilde\pi^*(P_\pi)\cong\{F\in\Jac(\widetilde X_s)\suchthat\cNm_{\pi}(F)\cong\mathcal O_X\}=\Prym_{\overline\pi}(\widetilde X_s).$$
We have thus proved the following result.
\begin{proposition}
  Suppose that the spectral curve $X_s \xrightarrow{\pi} X$ has $r_1$
  nodes with types $A_{m_i-1}$, $i=1,\dots,r_1$ and, that it has
  $r_2$ cusps of types $A_{m'_j-1}$, $j=1,\dots,r_2$. Let
  $\tilde\pi\colon \tilde{X}_s \to X_s$ be the normalization map and let
  $\bar{\pi} = \pi \circ \tilde{\pi}\colon \tilde{X}_s \to X$.
  Then there is a short exact
  sequence
  \begin{equation}\label{Prym singular}
    0\longrightarrow (\C^*)^{r_1}\times\C^{\sum_{i=1}^{r_1}(m_i-2)/2+\sum_{j=1}^{r_2} (m'_j-1)/2}\longrightarrow P_\pi\stackrel{\widetilde\pi^*}{\longrightarrow}\Prym_{\overline\pi}(\widetilde X_s)\longrightarrow 0,
  \end{equation}
  where $P_{\pi}$ was defined in (\ref{eq:def-P_pi}).
\end{proposition}

\section{Compactified Jacobian and parabolic modules}
\label{Compactified jacobian and GPLB}

Theorem \ref{BNR} tells us that when $X_s$ is singular, we must take
into account not only line bundles on $X_s$, but also rank $1$
torsion-free sheaves on $X_s$ (the case of a spectral curve with a single node goes back to Hitchin's paper \cite{hitchin:1987b}). The space of these objects with degree
$d+d_L$ provides a natural compactification
$\overline{\Jac^{d+d_L}(X_s)}$ of $\Jac^{d+d_L}(X_s)$. Furthermore, as
$X_s$ lies inside the complex surface given by the total space $T$ of
$L$, its singularities are planar. From \cite{altman-iarrobino-kleiman:1976} and \cite[Theorem A]{rego:1980},
this is equivalent to $\overline{\Jac^{d+d_L}(X_s)}$ being irreducible
and, therefore, $\Jac^{d+d_L}(X_s)$ is dense in
$\overline{\Jac^{d+d_L}(X_s)}$.  Of course, the compactification of
$\Jac^{d+d_L}(X_s)$ by rank $1$, degree $d+d_L$ torsion-free sheaves
is determined by the one of $\Jac(X_s)$ since all these spaces are
isomorphic.

We shall analyse here the effect of the compactification of
$\Jac(X_s)$ on the bundle 
$\widetilde\pi^*:\Jac(X_s)\to\Jac(\widetilde X_s)$ of 
Proposition~\ref{jacobian singular}, when $X_s$ has
singularities of type $A_{m-1}$.  

In order to do this we shall use \emph{parabolic modules}. These
objects were first defined by Rego \cite{rego:1980} and they were
intensively studied by Cook \cite{cook:1993,cook:1998}. Cook's work
generalizes to arbitrary rank and for any $A,D,E$ singularity
the work of Bhosle \cite{bhosle:1992} for ordinary nodes.

\subsection{Parabolic modules for one ordinary node}
For simplicity and motivation, let us start with the case of a unique ordinary node ($m=2$), where we have the sequence
\begin{equation}\label{eq:C*-bundle}
0\longrightarrow\C^*\longrightarrow\Jac(X_s)\stackrel{\widetilde\pi^*}{\longrightarrow}\Jac(\widetilde X_s)\longrightarrow 0
\end{equation}
which is a particular case of Proposition \ref{jacobian singular}.
Let $p\in X_s$ be the node and $\{p_1,p_2\}=\widetilde\pi^{-1}(p)$. The above sequence tells us that a line bundle $\mathcal{F}$ over $X_s$ is determined by a pair $(F,\lambda)$ consisting by a line bundle $F$ over the normalization $\widetilde X_s$ and a non-zero scalar $\lambda\in\C^*$, such that $\mathcal{F}|_{X\setminus\{p\}}\cong\widetilde\pi_*F|_{X\setminus\{p\}}$ and $\mathcal{F}_p$ is given by the identification of $F_{p_1}$ with $F_{p_2}$ via $\lambda$. In other words, $\mathcal{F}$ fits in the sequence
\begin{equation}\label{definition of cal L}
 0\longrightarrow\mathcal{F}\longrightarrow\widetilde\pi_*F\longrightarrow F_{p_1}\oplus F_{p_2}/U_\lambda(F)\longrightarrow0
\end{equation}
where $U_\lambda(F)$ is the $1$-dimensional subspace of $F_{p_1}\oplus F_{p_2}$ generated by $(1,\lambda)$.
Recall from (\ref{local ring for even singularities}) that $$\mathcal{O}_{X_s,p}\cong\{(f_1,f_2)\in\C_1[t_1]\oplus\C_2[t_2]\suchthat f_1(0)=f_2(0)\}.$$ Since $$\mathcal{F}_p\cong\{(s_1,s_2)\in\C_1[t_1]\oplus\C_2[t_2]\suchthat s_2(0)=\lambda s_1(0)\},$$ we see that 
$\mathcal{F}$ is indeed an invertible $\mathcal{O}_{X_s}$-module.

This motivates the consideration of all pairs $$(F,U_\lambda(F))$$ with $F\in\Jac(\widetilde X_s)$ and $U_\lambda(F)$ a $1$-dimensional subspace of $F_{p_1}\oplus F_{p_2}$.
If $U_\lambda(F)$ is generated by $(1,\lambda)$ with $\lambda\in\C^*$, then these are the pairs $(F,\lambda)$ mentioned in the previous paragraph. But now we are allowing all $1$-dimensional subspaces and these are parametrized by $\mathbb P^1$. The subspaces spanned by $(1,0)$ and $(0,1)$ correspond, respectively, to $\lambda=0$ and $\lambda=\infty$ in the compactification of $\C^*$ by $\mathbb P^1$, so we now allow $\lambda\in\C\cup\{\infty\}$. Such pairs were considered by Bhosle in \cite[Definition 2.1]{bhosle:1992} who called them \emph{generalized parabolic line bundles} on $\widetilde X_s$.
If $$\text{PMod}_2(\widetilde X_s)$$ denotes the moduli space of such
pairs $(F,U_\lambda(F))$ (cf.\ Remark~\ref{rem:cook-bhosle} below) then the
projection on the first
coordinate $$\mathrm{pr}_1:\text{PMod}_2(\widetilde
X_s)\longrightarrow\Jac(\widetilde X_s)$$ gives
$\text{PMod}_2(\widetilde X_s)$ the structure of a
$\mathbb{P}^1$-bundle over $\Jac(\widetilde X_s)$:
$$\mathbb{P}^1\longrightarrow\text{PMod}_2(\widetilde X_s)\stackrel{\mathrm{pr}_1}{\longrightarrow}\Jac(\widetilde X_s).$$

Again from $(F,U_\lambda(F))$ we construct a rank $1$ torsion-free sheaf $\mathcal{F}$ on $X_s$, by taking
$\mathcal{F}$ in the exact sequence
$$0\longrightarrow\mathcal{F}\longrightarrow\widetilde\pi_*F\longrightarrow F_{p_1}\oplus F_{p_2}/U_\lambda(F)\longrightarrow0.$$
If $\lambda\in\C^*$, this is precisely the construction in (\ref{definition of cal L}). If $\lambda=0$,
then $$\mathcal{F}_p\cong\{(s_1,s_2)\in\C_1[t_1]\oplus\C_2[t_2]\suchthat s_2(0)=0\}$$ and, if $\lambda=\infty$, then $$\mathcal{F}_p\cong\{(s_1,s_2)\in\C_1[t_1]\oplus\C_2[t_2]\suchthat s_1(0)=0\}.$$ In any case $\mathcal{F}_p$ is a $2$-dimensional $\mathcal{O}_{X_s,p}$-module and therefore $\mathcal F$ is not invertible.
So, using this construction, consider the map $$\tau:\text{PMod}_2(\widetilde X_s)\longrightarrow\overline{\Jac(X_s)}$$ defined by $$\tau(F,U_\lambda(F))=\mathcal F.$$

From \cite[Theorem 3]{bhosle:1992} we know that $\tau$ is a surjective morphism and $$\tau|_{\tau^{-1}(\Jac(X_s))}:\tau^{-1}(\Jac(X_s))\longrightarrow\Jac(X_s)$$ is an isomorphism. Hence, via $\tau|_{\tau^{-1}(\Jac(X_s))}$, the morphism $\mathrm{pr}_1|_{\tau^{-1}(\Jac(X_s))}$ is identified with $\widetilde\pi^*$ and the $\mathbb{P}^1$-bundle $\mathrm{pr}_1$ is a fibrewise compactification of the $\C^*$-bundle $\widetilde\pi^*$ in (\ref{eq:C*-bundle}):
$$\xymatrix{\mathbb{P}^1\ar[r]&\text{PMod}_2(\widetilde X_s)\ar[r]^{\ \mathrm{pr}_1}\ar[d]^{\tau}&\Jac(\widetilde X_s)\\
&\overline{\Jac(X_s)}.&&}$$

Moreover, the fibre of $\tau$ over $\overline{\Jac(X_s)}\setminus\Jac(X_s)$ consists of two points which are \emph{not} mapped, through $\mathrm{pr}_1$, to the same point of $\Jac(\widetilde X_s)$.
This shows that $\overline{\Jac(X_s)}$ does not fibre over
$\Jac(\widetilde X_s)$ through $\mathrm{pr}_1$ and $\tau$. In Example
\ref{intersection of fibres} below, we will develop this observation 
in a more general setting.

\subsection{Parabolic modules for $A_{m-1}$ singularities}
Let us now consider the generalization of this construction for the other kinds of singularities we are considering. 

\begin{definition}\label{def:parabolic modules for several nodes and cusps}
Suppose that $X_s$ has $r_1$ nodes $p_1,\ldots,p_{r_1}$, with $p_i$ of type $A_{m_i-1}$ with $m_i$ even and $r_2$ cusps $q_1,\ldots,q_{r_2}$, with $q_j$ of type $A_{m'_j-1}$ with $m'_j$ odd. For each $i=1,\ldots,r_1$ let  $\widetilde\pi^{-1}(p_i)=\{p_1^i,p_2^i\}$ and for each $j=1,\ldots,r_2$, let $\widetilde\pi^{-1}(q_j)=\{\tilde q_j\}$. A \emph{parabolic module on $\widetilde X_s$ on $X_s$} is a pair $(F,\underline U(F))$ where:
\begin{enumerate}
\item $F\in\Jac(\widetilde X_s)$;
\item $\underline U(F)=(U_1(F),\dots,U_{r_1}(F),U'_1(F),\ldots,U'_{r_2}(F))$ where:
\begin{itemize}
 \item for each $i=1,\ldots,r_1$, $U_i(F)$ is an $m_i/2$-dimensional subspace of the vector space $(F_{p_1^i}\oplus F_{p_2^i})^{m_i/2}$ which is also a $\mathcal{O}_{X_s,p_i}$-module via $\widetilde\pi_*$;
\item for each $j=1,\ldots,r_2$, $U'_j(F)$ is an $(m'_j-1)/2$-dimensional subspace of the vector space $F_{\tilde q_j}^{m'_j-1}$ which is also a $\mathcal{O}_{X_s,q_j}$-module via $\widetilde\pi_*$.
\end{itemize}
\end{enumerate}
\end{definition}


\begin{remark}
  \label{rem:cook-bhosle}
  Definition~\ref{def:parabolic modules for several nodes and cusps} is the special
  case of singularities of type $A_{m-1}$ of the general definition of Cook \cite{cook:1993,cook:1998}.  For ordinary nodes only, a parabolic
  module is a generalized parabolic line bundle in the sense of Bhosle
  in \cite{bhosle:1992}. Notice that, in the above definition, the
  condition for a subspace to be a $\mathcal{O}_{X_s,p}$-module via
  $\widetilde\pi_*$, is always satisfied for ordinary nodes. If we let $$\underline{m}=(m_1,\ldots,m_{r_1},m'_1,\ldots,m'_{r_2}),$$ then Cook constructed the moduli space $$\text{PMod}_{\underline{m}}(\widetilde X_s)$$ of parabolic modules on $\widetilde X_s$, associated to $r_1+r_2$ singularities of $X_s$ of type indexed by $\underline{m}$.
\end{remark}

From Proposition \ref{jacobian singular} we have the bundle
\begin{equation}\label{jacobian singular of type A_m-1}
(\C^*)^{r_1}\times\C^{\sum_{i=1}^{r_1}(m_i-2)/2+\sum_{j=1}^{r_2} (m'_j-1)/2}\stackrel{\delta}{\longrightarrow}\Jac(X_s)\stackrel{\widetilde\pi^*}{\longrightarrow}\Jac(\widetilde X_s).
\end{equation}
As in the case of a ordinary node, consider the projection on the first factor $$\mathrm{pr}_1:\text{PMod}_{\underline{m}}(\widetilde X_s)\longrightarrow\Jac(\widetilde X_s),\hspace{0,5cm}(F,\underline{U}(F))\mapsto F.$$ Cook showed that the fibre of this projection is a product $$\prod_{i=1}^{r_1}P(A_{m_i-1})\times\prod_{j=1}^{r_2}P(A_{m'_j-1})$$ of $r_1+r_2$ closed subschemes, $P(A_{m_i-1})$ or $P(A_{m'_j-1})$, of a certain Grassmannian (depending on the type of the corresponding singularity). These subschemes are, in general, quite complicated, but it is proved in \cite[Proposition 2, p.46]{cook:1998} that $P(A_{m-1})$ is \emph{connected} for every $m$.

Also, there is a finite morphism $$\tau:\text{PMod}_{\underline{m}}(\widetilde X_s)\longrightarrow\overline{\Jac(X_s)}$$ such that $\tau(F,\underline U(F))$ is the kernel of the restriction $$\widetilde\pi_*F\longrightarrow\displaystyle\bigoplus_{i=1}^{r_1}(F_{p_1^i}\oplus F_{p_2^i})^{m_i/2}/U_{i}(F)\oplus\bigoplus_{j=1}^{r_2}F_{\tilde q_j}^{m'_j-1}/U'_{j}(F).$$
The following is proved in Theorem 4.4.1 of \cite{cook:1993}.
\begin{proposition}
  The restriction $\tau|_{\tau^{-1}(\Jac(X_s))}$ gives an isomorphism
  $\tau^{-1}(\Jac(X_s))\cong\Jac(X_s)$.
\end{proposition}

We conclude that, under the identification of the previous proposition, $\mathrm{pr}_1|_{\tau^{-1}(\Jac(X_s))}$ is a \emph{fibrewise compactification} of the bundle $\widetilde\pi^*$ in (\ref{jacobian singular of type A_m-1}), 
\begin{equation}\label{bundle GP and morphism tau to compact Jac for general singularities}
\xymatrix{\displaystyle\prod_{i=1}^{r_1}P(A_{m_i-1})\times\prod_{j=1}^{r_2}P(A_{m'_j-1})\ar[r]&\text{PMod}_{\underline{m}}(\widetilde X_s)\ar[r]^{\hspace{0,2cm}\mathrm{pr}_1}\ar[d]^{\tau}&\Jac(\widetilde X_s)\\
&\overline{\Jac(X_s)}.&}
\end{equation}

In conclusion, one can say that \emph{$\mathrm{PMod}_m(\widetilde X_s)$ is a
compactification of $\Jac(X_s)$ which fibres over
$\Jac(\widetilde X_s)$}.  This should be contrasted with the
fact that the fibre of $\tau$ over
$\overline{\Jac(X_s)}\setminus\Jac(X_s)$ consists of a finite number
of points which may \emph{not} be mapped to the same point of
$\Jac(\widetilde X_s)$ by $\mathrm{pr}_1$. The following is an example
of this phenomenon. It will be important in the proof of Theorem
\ref{thm: fibre connected} below.
\begin{example}\label{intersection of fibres}
Let $m\geq 2$ be even, $p\in X_s$ be the only singularity of type $A_{m-1}$ and $\widetilde\pi^{-1}(p)=\{p_1,p_2\}$.
Consider the trivial bundle $\mathcal{O}_{\widetilde X_s}$ over
$\widetilde X_s$ and let $\C_{p_i}$ be its fibre over $p_i$. 
Let $$(\mathcal O_{\widetilde X_s},U_0(\mathcal O_{\widetilde X_s}))\in \text{PMod}_m(\widetilde X_s)$$ where $U_0(\mathcal O_{\widetilde X_s})$ is defined by
$$U_0(\mathcal O_{\widetilde X_s})=\{(v^1_1,v^1_2,\ldots,v^{m/2}_{m-1},v^{m/2}_m)\in(\C_{p_1}\oplus\C_{p_2})^{m/2}\suchthat v^{i/2}_i=0\text{ if }i\text{ is even}\}.$$ $U_0(\mathcal{O}_{\widetilde X_s})$ is then a subspace of $(\C_{p_1}\oplus \C_{p_2})^{m/2}$ which is also a $\mathcal{O}_{X_s,p}$-module and 
$$(\C_{p_1}\oplus \C_{p_2})^{m/2}/U_0(\mathcal{O}_{\widetilde X_s})=\C_{p_2}^{m/2}.$$

By definition, $\tau(\mathcal O_{\widetilde X_s},U_0(\mathcal O_{\widetilde X_s}))=\mathcal{F}$ fits in
\begin{equation}\label{L}
0\longrightarrow\mathcal{F}\longrightarrow\widetilde\pi_*\mathcal O_{\widetilde X_s}\longrightarrow\underbrace{\mathcal O_{\widetilde X_s,p_2}\oplus\dots\oplus\mathcal O_{\widetilde X_s,p_2}}_{m/2\text{ summands}}\longrightarrow0
\end{equation}
thus, over an open set $U\subset X_s$ which contains $p$,
\begin{equation}\label{L(U)}
\mathcal{F}(U)\cong\{(s_1,s_2)\in\C_1[t_1]\oplus\C_2[t_2]\suchthat s_2(0)=s_2'(0)=\dots=s_2^{((m-2)/2)}(0)=0\}.
\end{equation}

Now, consider the divisor $$E=\frac{m}{2}p_1-\frac{m}{2}p_2$$ over $\widetilde X_s$, the line bundle $\mathcal{O}_{\widetilde X_s}(E)$ and let $$(\mathcal O_{\widetilde X_s}(E),U_\infty(\mathcal O_{\widetilde X_s}(E))\in\text{PMod}_m(\widetilde X_s),$$ where
$U_\infty(\mathcal O_{\widetilde X_s}(E))$ is given by
$$U_\infty(\mathcal O_{\widetilde X_s}(E))=\{(v^1_1,v^1_2,\ldots,v^{m/2}_{m-1},v^{m/2}_m)\in(\C_{p_1}\oplus\C_{p_2})^{m/2}\suchthat v^{i/2}_{i-1}=0\text{ if }i\text{ is even}\}.$$ Then $U_\infty(\mathcal O_{\widetilde X_s}(E))$ is a subspace of $(\C_{p_1}\oplus \C_{p_2})^{m/2}$ which is also a $\mathcal{O}_{X_s,p}$-module and 
$$(\C_{p_1}\oplus \C_{p_2})^{m/2}/U_\infty(\mathcal O_{\widetilde X_s}(E))=\C_{p_1}^{m/2}.$$
So, if we consider $\mathcal{F}'$ such that it fits in
\begin{equation}\label{L'}
0\longrightarrow\mathcal{F}'\longrightarrow\widetilde\pi_*\mathcal O_{\widetilde X_s}\longrightarrow\underbrace{\mathcal O_{\widetilde X_s,p_1}\oplus\dots\oplus\mathcal O_{\widetilde X_s,p_1}}_{m/2\text{ summands}}\longrightarrow0,
\end{equation}
then, over $U$,
\begin{equation}\label{L'(U)}
\mathcal{F}'(U)\cong\{(s_1,s_2)\in\C_1[t_1]\oplus\C_2[t_2]\suchthat s_1(0)=s_1'(0)=\dots=s_1^{((m-2)/2)}(0)=0\}.
\end{equation}
Tensoring the sequence (\ref{L'}) by
$\widetilde\pi_*\mathcal{O}_{\widetilde X_s}(E)$ we conclude
that 
$$\tau(\mathcal O_{\widetilde X_s}(E),U_\infty(\mathcal O_{\widetilde X_s}(E)))=
\mathcal{F}'\otimes\widetilde\pi_*\mathcal{O}_{\widetilde X_s}(E).$$ 
Now, from (\ref{L(U)}) and (\ref{L'(U)}) one
has $$\mathcal{F}'\otimes\widetilde\pi_*\mathcal{O}_{\widetilde
  X_s}(E)\cong\mathcal{F}$$ and thus
\begin{equation}\label{does not fibre}
\tau(\mathcal O_{\widetilde X_s},U_0(\mathcal O_{\widetilde X_s}))=\tau(\mathcal O_{\widetilde X_s}(E),U_\infty(\mathcal O_{\widetilde X_s}(E)))=\mathcal F. 
\end{equation}
But 
$$\mathrm{pr}_1(\mathcal O_{\widetilde X_s},U_0(\mathcal O_{\widetilde X_s}))=\mathcal O_{\widetilde X_s}$$
while 
$$\mathrm{pr}_1(\mathcal O_{\widetilde X_s}(E),U_\infty(\mathcal O_{\widetilde X_s}(E)))=\mathcal O_{\widetilde X_s}(E).$$
Hence $\overline{\Jac(X_s)}$ does not fibre over 
$\Jac(\widetilde X_s)$ via $\mathrm{pr}_1$ and $\tau$.
\end{example}

\section{The fibre of the Hitchin map for irreducible singular
  spectral curve }\label{The singular case}
In the two preceding sections we have been analysing the structure of
the Jacobian of $X_s$ and its compactification by rank $1$
torsion-free sheaves. Now we come back to our goal: the description of
the fibre $\cH^{-1}(s)$.

For the remainder of this section, let $s\in H^0(X,L^2)$ be such that
the spectral curve $X_s$ is singular and irreducible. Thus
$D_s=\divisor(s)$ has at least one multiple point, but $s$ is not a
square of a section of $L$ (cf.\
Remark~\ref{rem:singular-smooth-X_s}).

\subsection{Description of the fibre}

As already mentioned, in order to have a description of the fibre of
$\cH$ over $s$, and taking into account Theorem \ref{BNR} and
Proposition \ref{fibre smooth case = Prym}, we now have to consider
also rank $1$ torsion-free sheaves over $X_s$. Let
$\overline{P_\pi}\subset\overline{\Jac(X_s)}$ denote the
compactification of $P_\pi$ obtained by taking its closure inside
$\overline{\Jac(X_s)}$. Then $P_{\pi}$ is dense in $\overline{P_\pi}$.

Consider also the closure $\overline{\mathcal T_{\pi}}\subset\overline{\Jac^{d+d_L}(X_s)}$ of $\mathcal T_{\pi}$, defined in (\ref{T_pi}), induced by the compactification of $\Jac^{d+d_L}(X_s)$. Again, $\overline{\mathcal T_{\pi}}$ is a torsor for $\overline{P_\pi}$, the isomorphism $\overline{\mathcal T_{\pi}}\to\overline{P_\pi}$ being given by $\mathcal F\mapsto\mathcal F\otimes L_0^{-1}$ where $L_0$ is a fixed element of $\mathcal T_{\pi}$.

\begin{theorem}\label{fibre of Hitchin map}
Let $s\in H^0(X,L^2)$ such that $X_s$ is singular and irreducible. Then $\cH^{-1}(s)$ is isomorphic to $\overline{\mathcal T_{\pi}}$.
\end{theorem}
\proof
From the definition of $\mathcal T_{\pi}$, it is clear that $\mathcal T_{\pi}\subset \cH^{-1}(s)$. So elements in $\mathcal T_{\pi}$ are in one to one correspondence with $L$-twisted Higgs pairs $(V,\varphi)$ such that $\Lambda^2V\cong\Lambda$, $\det(\varphi)=s$ and $\tr(\varphi)=0$. All these are closed conditions, therefore if $\cF\in\overline{\mathcal T_{\pi}}$, the $L$-twisted Higgs pair $(V,\varphi)$ obtained from $\cF$ will also satisfy the same conditions. Finally, notice that, as in the smooth case, there cannot exist a $\varphi$-invariant line subbundle $N\subset V=\pi_*\mathcal F$, because that would contradict the irreducibility of $X_s$. This ensures that $(V,\varphi)$ is stable, hence $\overline{\mathcal T_{\pi}}\subset \cH^{-1}(s)$.

Conversely, let $(V,\varphi)\in \cH^{-1}(s)$. It is then identified with a rank $1$ torsion-free sheaf $\mathcal F\in\overline{\Jac^{d+d_L}(X_s)}$ such that $\det(\pi_*\mathcal F)\cong\Lambda$. Now, let $\mathcal F'\in\overline{\Jac(X_s)}$ represent the point corresponding to $\mathcal F$ under the isomorphism $\overline{\Jac(X_s)}\simeq\overline{\Jac^{d+d_L}(X_s)}$.
Then $\mathrm{pr}_1(\tau^{-1}(\mathcal F'))\subset\Prym_{\overline\pi}(\widetilde X_s)$ where $\mathrm{pr}_1$ and $\tau$ are the morphisms in (\ref{bundle GP and morphism tau to compact Jac for general singularities}). But $\mathrm{pr}_1$ is a bundle which is a fibrewise compactification of the bundle $\widetilde\pi^*$, so, when suitably restricted, it also compactifies the bundle (\ref{Prym singular}).
Hence $\mathcal F'\in\overline{P_\pi}$ i.e. $\mathcal F\in\overline{\mathcal T_{\pi}}$.
\endproof

The dimension of the fibre of $\cH$ in the case treated here can be easily computed.

\begin{proposition}
\label{prop:dimension singular fibre irreducible spectral curve} 
Let $s\in H^0(X,L^2)$ be such that $X_s$ is irreducible. 
Then $\dim(\cH^{-1}(s))=d_L+g-1$.
\end{proposition}
\proof If $X_s$ is smooth, this is just (\ref{dimension of generic
  fibre}). Thus it suffices to consider the case when $X_s$ is
singular and irreducible.  
The ramification divisor of
$\overline\pi:\widetilde X_s\to X$ is $\widetilde
D_s=\sum_{j=1}^{r_2}q_j$. Therefore, by the Riemann-Hurwitz' formula, the genus of $\widetilde X_s$ is
\begin{equation}\label{genus of normalization}
 g(\widetilde X_s)=2g-1+r_2/2,
\end{equation}
where we recall that $r_2$ is even (see Remark \ref{r2 even}).

From (\ref{Prym singular}), $$\dim\overline{P_\pi}=\dim P_\pi=\dim\Prym_{\overline\pi}(\widetilde X_s)+\sum_{i=1}^{r_1}m_i/2+\sum_{j=1}^{r_2}(m'_j-1)/2$$
and $\dim\Prym_{\overline\pi}(\widetilde X_s)=g(\widetilde X_s)-g$. 
Hence from (\ref{genus of normalization}), $$\dim\Prym_{\overline\pi}(\widetilde X_s)=g-1+r_2/2.$$ 
Therefore, from Theorem~\ref{fibre of Hitchin map}, the dimension of
the fibre of $\cH$ over $s$ is
$$\dim\overline{\mathcal T_{\pi}}=\dim\overline{P_\pi}=g-1+\frac{1}{2}\left(\sum_{i=1}^{r_1}m_i+\sum_{i=1}^{r_2}m'_j\right)=d_L+g-1.$$
\endproof

\subsection{Connectedness of the fibre}
Now we can prove that the fibre of $\cH$ is connected.

\begin{theorem}\label{thm: fibre connected}
  Assume $\deg(L)>0$ and let $s\in H^0(X,L^2)$ be such that $X_s$ is
  irreducible. Then the fibre of $\cH:\cM_L^\Lambda\to H^0(X,L^2)$
  over $s$ is connected.
\end{theorem}
\proof 
If $D_s=\divisor(s)$ has no multiple points, $\cH^{-1}(s)=\mathcal
T_{\pi}\cong\Prym_{\pi}(X_s)$ hence connected. (Here we use that $X_s\to X$
is ramified because $\deg(L)>0$, cf.\ Remark~\ref{rem:prym-components}.)

Suppose now that $D_s$ has multiple points.  From (\ref{bundle GP and
  morphism tau to compact Jac for general singularities}) we have the
morphisms $\tau$ and $\mathrm{pr}_1$
\begin{equation}\label{bundle GP and morphism tau to compact Prym for general singularities}
 \xymatrix{\tau^{-1}(\overline{P_\pi})\ar[r]^{\mathrm{pr}_1}\ar[d]^{\tau}&\Prym_{\overline\pi}(\widetilde X_s)\\
\overline{P_\pi}&}.
\end{equation} 

Assume that $\Prym_{\overline\pi}(\widetilde X_s)$ is connected. This occurs if and only if the cover $\overline\pi:\widetilde X_s\to X$ is ramified, which is equivalent to saying that $D_s$ has some point with odd multiplicity.
From sequence (\ref{Prym singular}), $P_\pi$ is connected hence $\overline{P_\pi}$ is connected, because
$P_\pi$ is dense in $\overline{P_\pi}$ (recall that $\Jac(X_s)$ is dense in $\overline{\Jac(X_s)}$).

Suppose now that $\Prym_{\overline\pi}(\widetilde X_s)$ is not
connected, i.e., that $D_s$ is twice another divisor.  In
\cite{mumford:1971} Mumford shows that
$\Prym_{\overline\pi}(\widetilde X_s)$ has two components, which we
denote by $\Prym_{\overline\pi}^0(\widetilde X_s)$ and
$\Prym_{\overline\pi}^1(\widetilde X_s)$, and that the difference
between them lies in the parity of the dimension of the space of
sections:
$$\Prym_{\overline\pi}^0(\widetilde X_s)=\{F\in\Prym_{\overline\pi}(\widetilde X_s)\suchthat\dim H^0(\widetilde X_s,F\otimes\overline{\pi}^*F_0)\text{ even}\}$$
and
$$\Prym_{\overline\pi}^1(\widetilde X_s)=\{F\in\Prym_{\overline\pi}(\widetilde X_s)\suchthat\dim H^0(\widetilde X_s,F\otimes\overline{\pi}^*F_0)\text{ odd}\}$$
where $F_0$ is a fixed square root of the canonical line bundle $K$ of $X$.
Moreover, given $q\in\widetilde X_s$ and
\begin{equation}\label{LinPrym0}
F\in\Prym_{\overline\pi}^i(\widetilde X_s)
\end{equation}
then
\begin{equation}\label{L(q-sigmq)inPrym1}
F(q-\tilde\sigma(q))\in\Prym_{\overline\pi}^{1-i}(\widetilde X_s),
\end{equation}
where $\tilde\sigma:\widetilde X_s\to\widetilde X_s$ is the involution exchanging the sheets of the double cover $\overline\pi:\widetilde X_s\to X$.

Returning to $P_\pi$, we see from (\ref{Prym singular}) that $P_\pi$
has two components, because it is a bundle over a space with two
connected components with connected fibre. Moreover, since
$\mathrm{pr}_1$ in (\ref{bundle GP and morphism tau to compact Prym
  for general singularities}) is a bundle which compactifies fibrewise
$\widetilde\pi^*$ in (\ref{Prym singular}), the space 
$\tau^{-1}(\overline{P_\pi})$ also has two connected components. We
shall use the construction of Example~\ref{intersection of fibres} to
show that the images under $\tau$ of these two components intersect in
$\overline{P_{\pi}}$. We begin by considering two special cases.

\subsubsection*{Case 1}

Suppose that $D_s=mp$ with
\begin{math}\label{m=2mod4}
m=2 \mod 4.
\end{math}
Let
\begin{displaymath}
  F_0 = \mathrm{pr}_1^{-1}(\mathcal O_{\widetilde{X}_s})
  \qquad\text{and}\qquad
  F_1 = \mathrm{pr}_1^{-1}(\mathcal O_{\widetilde X_s}((m/2)p_1-(m/2)p_2)).
\end{displaymath}
It follows from Example~\ref{intersection of fibres} that $\tau(F_0)$
intersects $\tau(F_1)$ because, by (\ref{does not fibre}), we have
that $\mathcal F$ defined in (\ref{L}) belongs to $\tau(F_0)\cap
\tau(F_1)$.  Now, since $\widetilde\pi^{-1}(p)=\{p_1,p_2\}$, we have
$\tilde\sigma(p_1)=p_2$. Furthermore $ \mathcal O_{\widetilde
  X_s}\in\Prym_{\overline\pi}^0(\widetilde X_s) $. Thus it follows
from (\ref{LinPrym0}), (\ref{L(q-sigmq)inPrym1}) and (\ref{m=2mod4})
that $$\mathcal O_{\widetilde
  X_s}((m/2)p_1-(m/2)p_2)\in\Prym_{\overline\pi}^1(\widetilde X_s).$$
The conclusion is that the images under $\tau$ of the two components
of $\tau^{-1}(\overline{P_\pi})$ intersect. Therefore
$\overline{P_\pi}$ is connected and hence the same holds for
$\overline{\mathcal T_{\pi}}$.

\subsubsection*{Case 2}

Suppose that $D_s=mp$ with $m=0\mod 4$.  Take again
$F_0=\mathrm{pr}_1^{-1}(\mathcal O_{\widetilde X_s})$, but now
consider $(\mathcal O_{\widetilde X_s}(p_1-p_2),U(\mathcal
O_{\widetilde X_s}(p_1-p_2)))\in F_1=\mathrm{pr}_1^{-1}(\mathcal
O_{\widetilde X_s}(p_1-p_2))$, where
\begin{multline*}
  U(\mathcal O_{\widetilde{X}_s}(p_1-p_2))=
  \{(v_1^1,v_2^1,\ldots,v_{m-1}^{m/2},v_m^{m/2})
    \in(\C_{p_1}\oplus\C_{p_2})^{m/2} \suchthat \\ 
  v_1^1=0 \text{ and }v_i^{i/2}=0,\text{ for } 2\leq i\leq m-2\text{ even}\}.
\end{multline*}
Then $U(\mathcal O_{\widetilde X_s}(p_1-p_2))$ is an $m/2$-dimensional
subspace of $(\C_{p_1}\oplus\C_{p_2})^{m/2}$ and it is an $\mathcal
O_{X_s,p}$-module.  Using very similar arguments to those of Example
\ref{intersection of fibres}, one sees again that $\tau(F_0)$
intersects $\tau(F_1)$. Moreover, by (\ref{L(q-sigmq)inPrym1}), $F_0$
and $F_1$ are fibres over different components of
$\Prym_{\overline\pi}(\widetilde X_s)$. Hence the images under $\tau$
of the two components of $\tau^{-1}(\overline{P_\pi})$ intersect and
we conclude that $\overline{P_\pi}$ and $\overline{\mathcal T_{\pi}}$
are connected.

The general case is proved by using the appropriate local construction
at each of the singularities to construct $F_0$ and $F_1$
contained in different components of $\tau^{-1}(\overline{P}_{\pi})$
whose images under $\tau$ intersect in $\overline{P}_{\pi}$.
\endproof

\section{The fibre of the Hitchin map for reducible spectral curve}\label{The non-integral case}

In this section we shall analyse the fibre of
$\cH:\cM_L^\Lambda\to H^0(X,L^2)$ over the sections $s$
such that the spectral curve $X_s$ is reducible. We shall resort to
a direct method for describing the fibre $\cH^{-1}(s)$.  We start by
constructing certain fibre bundles over the Jacobian, which will be
used in the analysis of the fibre of the Hitchin map. This is carried
out in Section~\ref{sec:strat-fibre-hitchin}, where we also state a
structure theorem about the fibre. The proof of this theorem takes up
the next two sections. In Section~\ref{sec:proof-lemma-morphism} we
construct families of Higgs pairs over these fibre bundles, thus
obtaining morphisms from their total spaces to the fibre. Having
established that, in Section~\ref{sec:proof-lemma-surjective} we show
that these families give rise to a stratification of the
fibre. Finally, in
Section~\ref{sec:connectedness-dimension-fibre-redicble-case}, we put
the results of the preceding sections together to obtain the
connectedness and dimension results for the fibre in the case of
reducible spectral curve.

Recall from Remark~\ref{rem:singular-smooth-X_s} that $X_s$ is reducible if
and only if the section $s\in H^0(X,L^2)$ admits a square root $s'\in
H^0(X,L)$. Throughout this section we fix such a
square root. Thus we have
\begin{equation}\label{widetildeD}
D_s=2 D',
\end{equation}
and $L\cong\mathcal{O}(D')$, where where $D_s$ and $D'$ are the
(effective) divisors of $s$ and $s'$ respectively. 

\subsection{Stratification of the fibre of the Hitchin map}
\label{sec:strat-fibre-hitchin}

First we introduce some notation.  
For any effective divisor $D$ on $X$
and any line bundle $M\in\Jac^m(X)$ of
degree $m$, 
define subspaces of $H^0(D',M^2L\Lambda^{-1})$ as follows:
\begin{align}\label{E(D,M)}
  E(D,M)&=
  \bigl\{q \in H^0(D',M^2L\Lambda^{-1})\;|\;
  q_{|D'-D} = 0\}, \\
  F(D,M)&=
  \bigl\{q \in H^0(D',M^2L\Lambda^{-1})\;|\;
  \begin{cases}
      \mathrm{ord}_p(q)=D'(p)- D(p) &\text{if
      $0 < D(p) \leq D'(p)$}\\
    q(p)=0 & \text{otherwise}
  \end{cases}
  \bigr\}.
  \label{subspace cal W}
\end{align}

\begin{remark}\label{restriction to a divisor 1}
  Recall that if $\Delta=\sum n_ip_i$ is any effective divisor in $X$,
  then, choosing a local coordinate $z_i$ centred at $p_i$, a global
  section of $\mathcal{O}_{\Delta}$ can be written as $\sum f_i(z)$ where
  $f_i(z)=\sum_{k=0}^{n_i-1}a_kz_i^k$. Thus, if we choose a local
  coordinate $z$ around each $p\in\Supp(D')$, the space $E(D,M)$
  consists of sections of the form
\begin{displaymath}
\sum_{p\in\Supp(D')}\sum_{k=D'(p)-D(p)}^{ D'(p)-1}a_kz^k
\end{displaymath}
and the space $F(D,M)$ consists of such sections with 
$a_{D'(p)-D(p)}\neq 0$ (we interpret an empty sum as being equal to zero).
\end{remark}

Note that the space $E(D,M)$ is a linear subspace of
$H^0(D',M^2L\Lambda^{-1})$, while $F(D,M)$ does not
contain zero unless $D = 0$, in which case $F(D,M)=0$. We gather some
obvious observations about these spaces in the following Proposition.

\begin{proposition}
  \label{prop:properties-EDM}
    (1) The spaces $E(D,M)$ give a filtration of
    $H^0(D',M^2L\Lambda^{-1})$ indexed by divisors $D$ satisfying
    $0\leq D\leq D'$. Thus
    \begin{displaymath}
      D_1 \leq D_2 \;\implies\; E(D_1,M) \subseteq E(D_2,M).
    \end{displaymath}
    Moreover, $E(D,M)=0$ if and only if
    $D=0$, and $E(D,M) = H^0(D',M^2L\Lambda^{-1})$ if and only if
    $D\geq D'$.

    (2) The space $E(D,M)$ is the disjoint union
    \begin{displaymath}
      E(D,M) = \bigcup_{0\leq\Dbar\leq D}F(\Dbar,M).
    \end{displaymath}

    (3) For any two effective divisors $D_1$ and $D_2$ on $X$, we have
    \begin{align*}
      E(D_1,M) \cap E(D_2,M) &= E(\min\{D_1,D_2\},M), \\
      E(D_1,M) \cup E(D_2,M) &= E(\max\{D_1,D_2\},M).
    \end{align*}

\end{proposition}

\begin{remark}
  \label{rem:C-star-action-FDM}
  If $D\neq 0$ then $F(D,M)$ is a
  linear subspace of $H^0(D',M^2L\Lambda^{-1})$ with a hyperplane
  removed.  Thus $\C^*$ acts by multiplication on $F(D,M)$ and
  \begin{equation}\label{eq:dimcal W/C}
    \dim F(D,M)/\C^*=\deg(D)-1.
  \end{equation}
\end{remark}

\begin{definition}\label{bundle E(D,m)}
  Let $D$ be a divisor satisfying $0 \leq {D} \leq D'$ and let $m$ be
  an integer. We denote by 
  \begin{equation}
    \mathcal{E}(D,m)\longrightarrow\Jac^m(X)
  \end{equation}
  the vector bundle (constructed using the degree $m$ Poincar\'e
  line bundle) whose fibre over $M\in\Jac^m(X)$ is $E(D,M)$ and by
  \begin{equation}
    \label{cal E i,l}
    \mathcal{F}(D,m)\longrightarrow\Jac^m(X)
  \end{equation}
  the subbundle\footnote{Note that this bundle is not a vector
    bundle.}, whose fibre over $M$ is $F(D,M)$. In particular,
  \begin{math}
    \mathcal F(0,m) = \mathcal E(0,m) = \Jac^{m}(X).
  \end{math}

\end{definition}

\begin{remark}
  \label{rem:properties-EDM}
  Clearly the properties stated in
  Proposition~\ref{prop:properties-EDM} and
  Remark~\ref{rem:C-star-action-FDM} give rise to analogous properties
  for $\mathcal{E}(D,m)$ and $\mathcal{F}(D,m)$.
\end{remark}

\begin{remark}
  Notice that, if $d$ is even, the only compact $\mathcal E(D,m)$ is
  $\mathcal E(0,d/2)$, whereas if $d$ is odd, the compact $\mathcal
  E(D,m)$ are those of the form $\mathcal E(D,(d-1)/2)$, i.e.\ those
  such that $\deg(D)=1$.
\end{remark}

Now we state the relation between the spaces $\mathcal{E}(D,m)$ and
$\mathcal{F}(D,m)$ just introduced and the fibre of the Hitchin map.

\begin{theorem}\label{prop:fibre when spectral curve is reducible}
  Let $s\in H^0(X,L^2)\setminus\{0\}$ be such that $D_s=2 D'$ and
  $L=\mathcal{O}(D')$. For any integer $m$ and any effective divisor
  $D \leq D'$ such that
  \begin{equation}\label{eq:bound-m}
    d/2-\deg(D) \leq m \leq d/2,
  \end{equation}
  there is a morphism
  \begin{math}
    p\colon \mathcal{E}(D,m) \to \mathcal{H}^{-1}(s)
  \end{math}
  with the following properties:
  \begin{enumerate}
  \item The union of the images
    \begin{math}
      p(\mathcal{E}(D,m))
    \end{math}
    over all $(D,m)$ satisfying the above conditions covers the fibre
    $\mathcal{H}^{-1}(s)$.
  \item If $m_1$ satisfies (\ref{eq:bound-m}) for a fixed divisor $D$,
    then so does $m_2 = d - \deg(D) - m_1$ and
    \begin{displaymath}
      p(\mathcal{F}(D,m_1)) = p(\mathcal{F}(D,m_2)).
    \end{displaymath}
  \item If $D \neq 0$, then there is a fibrewise $\C^*$-action on
    $\mathcal{F}(D,m)$ and the restriction of $p$ factors through the
    quotient to induce a morphism
    \begin{displaymath}
      p\colon\mathcal{F}(D,m)/\C^* \longrightarrow \mathcal{H}^{-1}(s).
    \end{displaymath}
    This morphism is an isomorphism onto its image, unless $m =
    (d-\deg(D))/2$. In the latter case it is generically two to one,
    ramified at the preimage in $\mathcal{F}(D,m)$ of the locus of
    line bundles $M \in \Jac^m(X)$ satisfying $M^2 \cong \Lambda(-D)$.
  \item If $D = 0$, then $m=d/2$ and the restriction of $p$ to
    $\mathcal{F}(0,d/2)$ 
    \begin{displaymath}
      p\colon\mathcal{F}(0,d/2) \longrightarrow \mathcal{H}^{-1}(s)
    \end{displaymath}
    is generically two to one, ramified at the locus of line bundles
    $M \in \Jac^{d/2}(X) = \mathcal{F}(0,d/2)$ satisfying $M^2 \cong
    \Lambda$.
\end{enumerate}
\end{theorem}

\begin{proof}
  The proof takes up the next two sections. In
  Section~\ref{sec:proof-lemma-morphism} we construct the morphism $p$
  and show that it descends to the quotient by $\C^*$ when $D \neq 0$
  (cf.\ Remarks~\ref{rem:C-star-action-FDM} and \ref{rem:properties-EDM}).
  In Section~\ref{sec:proof-lemma-surjective} we prove the
  remaining properties stated.
\end{proof}


\begin{remark}\label{dependence of p on square root of s}
  As will be seen from the construction, the morphism $p$ depends on
  the choice of the square root $s'$ of $s$.
\end{remark}

There is a certain redundancy in the description of the fibre of
Theorem~\ref{prop:fibre when spectral curve is reducible}. Indeed,
note that in (2) of the theorem, we have $m_1=m_2$ if and only if $m_1 =
(d-\deg(D))/2$. Thus we may, by using the larger of the $m_i$, write
\begin{displaymath}
  \mathcal{H}^{-1}(s) = \bigcup p(\mathcal{E}(D,m))
  = \bigcup p(\mathcal{F}(D,m)),
\end{displaymath}
where the union is over $(D,m)$, with $D$ an effective divisor
  $D \leq D'$ such that
  \begin{equation}\label{eq:bound-m-largest}
    (d-\deg(D))/2 \leq m \leq d/2.
  \end{equation}


\subsection{Construction of the morphism $p$}
\label{sec:proof-lemma-morphism}

Any element of $\mathcal{E}(D,m)$ is given by a pair
\begin{math}
(q,M),
\end{math} 
where $M\in\Jac^m(X)$ and
\begin{equation}\label{eq:[q]}
  q\in E(D,M).
\end{equation}
In order to construct a pair $(V,\varphi)\in\cH^{-1}(s)$ from this
data, we shall make use of the short exact sequence of complexes of sheaves
$0 \to C^\bullet_1 \to C^\bullet_2 \to C^\bullet_3 \to 0$ given by the diagram
\begin{equation}\label{eq:ses-complexes}
\xymatrix{&0\ar[d]&0\ar[d]\\
C^\bullet_1:&M^2\Lambda^{-1}\ar[r]^(.4){\Id}\ar[d]_{=}&M^2\Lambda^{-1}\ar[d]^{c}&\\
C^\bullet_2:&M^2\Lambda^{-1}\ar[r]^(.5){c}\ar[d]&M^2L\Lambda^{-1}\ar[d]^{r( D')}&\\
C^\bullet_3:&0\ar[r]^(.4){0}\ar[d]&M^2L\Lambda^{-1}|_{ D'}\ar[d]&\\
&0&0,}\vspace{0,5cm}
\end{equation}
where the map $c$ is defined by $\psi\mapsto\sqrt{-1}\,s'\psi$ and $r(D')$ denotes the restriction to the divisor $D'$.
The associated long exact sequence in hypercohomology
shows that $r(D')$ yields an isomorphism
$$r(D')\colon\mathbb{H}^1(X,C^\bullet_2)\xrightarrow{\cong}\mathbb{H}^1(X,C^\bullet_3)=H^0( D',M^2L\Lambda^{-1}).$$
With respect to some open covering $U=\{U_a\}$ of $X$, choose a
representative $(x_{ab},y_a)$ of the class $r( D')^{-1}(q)\in
\mathbb{H}^1(X,C^\bullet_2)$ corresponding to $q$ of
(\ref{eq:[q]}). 
Then
\begin{equation}\label{eq:gluing condition}
\sqrt{-1}\,s' x_{ab}=y_b-y_a
\end{equation}
and $y_a|_{ D'}=q|_{ D'\cap U_a}$. We shall construct a pair $(V,\varphi)$ from $(x_{ab},y_a)$ as
follows. We let $V$ be the vector bundle defined by taking on each
open $U_a$ the direct sum
\begin{equation}\label{reconstruction of V 1 2}
M|_{U_a}\oplus M^{-1}\Lambda|_{U_a}
\end{equation}
and gluing over $U_{ab}$ via the map
\begin{equation}\label{reconstruction of V 2 2}
f_{ab}=\begin{pmatrix}
 1_M & x_{ab}/2 \\
 0 & 1_{M^{-1}\Lambda}
\end{pmatrix}.
\end{equation}
Also over each open $U_a$, consider the section of
$H^0(U_a,\End_0(M\oplus M^{-1}\Lambda)\otimes L)$ given, with respect
to the decomposition (\ref{reconstruction of V 1 2}), by
\begin{equation}\label{varphi local}
\varphi_a=\begin{pmatrix}
 \sqrt{-1}\,s' & y_a \\
 0 & -\sqrt{-1}\,s'
\end{pmatrix}.
\end{equation}
From (\ref{eq:gluing condition}), one has $f_{ab}\varphi_b=\varphi_a
f_{ab}$, so $\{\varphi_a\}$ gives a global traceless endomorphism
$\varphi:V\to V\otimes L$.

\begin{remark}
  The long exact cohomology sequence associated to the
  vertical short exact sequence on the right of
  (\ref{eq:ses-complexes}) gives an isomorphism
  $\mathbb{H}^1(X,C^\bullet_2) \cong H^1(X,M^2\Lambda^{-1})$.
  The vector bundle $V$ is just the extension
  \begin{displaymath}
    0 \longrightarrow M \longrightarrow V \longrightarrow M^{-1}\Lambda \longrightarrow 0
  \end{displaymath}
  given by the image of $r(D')^{-1}(q)$ under this isomorphism. The
  point of the preceding construction using hypercohomology is that it
  provides a convenient way of encoding the construction of $\varphi$.

  Note also that when $D = 0$, we have $q=0$ which gives rise to
  $V=M\oplus M^{-1}\Lambda$ and $\varphi=\begin{pmatrix}
    \sqrt{-1}\,s' & 0\\
    0 & -\sqrt{-1}\,s'
  \end{pmatrix}$.
\end{remark}

We have thus constructed a Higgs pair $(V,\varphi)$ with $\det(\varphi) =
s$. Moreover, by construction 
\begin{displaymath}
  M = \ker(\varphi - s') \subset V.
\end{displaymath}
It remains to prove that $(V,\varphi)$ is semistable. For that we need
the following obvious observation.
\begin{proposition}
  \label{prop:M-plus-minus}
  Let $(V,\varphi)$ be an $L$-twisted, rank $2$ Higgs pair with $\Lambda^2V=\Lambda$. Assume that $\det(\varphi) \in
  H^0(X,L^2)$ has a square root $s' \in H^0(X,L)$. Let
  \begin{math}
    M_{\pm} = \ker(\varphi\pm s')
  \end{math}
  and let $D=\divisor(\epsilon)$, where $\epsilon\colon M_{+}M_{-}\to
  \Lambda$ is induced by the injective sheaf map $M_{+}\oplus M_{-}
  \into V$. Then
  \begin{displaymath}
    M_{+}M_{-} = \Lambda(-D).
  \end{displaymath}
  Moreover, $M_{+}$ and $M_{-}$ are the only $\varphi$-invariant
  subbundles of $V$.
\end{proposition}
In view of this proposition, we only need to check that the
semistability condition holds for the $\varphi$-invariant subbundles
$M$ and $M^{-1}\Lambda(-D)$. But this is equivalent to the assumption
(\ref{eq:bound-m}).

We note that our construction can clearly be carried out in
families, and thus the set theoretic map
$p:\mathcal{E}(s)\to\cH^{-1}(s)$ given by
$$p([q],M)=\text{ isomorphism class of }(V,\varphi)\text{ defined
  above}$$ is in fact a morphism.

Finally, in order to see that when $D\neq 0$, the morphism $p\colon
\mathcal{F}(D,m)\to \mathcal{H}^{-1}(s)$ descends to the quotient
$\mathcal{F}(D,m)/\C^*$ proceed as follows: if we carry out the above
construction with $\beta q$ for some $\beta\in\C^*$ we obtain a pair
$(\tilde V, \tilde \varphi)$ and we get an isomorphism $g\colon
(V,\varphi) \to (\tilde V,\tilde \varphi)$ by defining locally, with
respect to the decomposition $M|_{U_a}\oplus M^{-1}\Lambda|_{U_a}$,
$$g_a=\begin{pmatrix}
  \sqrt{\beta^{-1}} & 0\\
  0 & \sqrt{\beta}
\end{pmatrix}$$
for a choice a square root $\sqrt{\beta}$.

\subsection{Conclusion of the proof of Theorem~\ref{prop:fibre when
    spectral curve is reducible}}
\label{sec:proof-lemma-surjective}
We start by some preliminary observations and results.
Let  $(V,\varphi)\in \cH^{-1}(s)$, with $s\in
H^0(X,L^2)\setminus\{0\}$ such that $X_{s}$ is
reducible. Then $(V,\varphi)$ is a rank $2$ semistable $L$-twisted
Higgs pair, such that $\Lambda^2V\cong\Lambda$, $\tr(\varphi)=0$ and
$\det(\varphi)=s$.  Thus $\varphi$ has
eigenvalues 
$$\pm s'\in H^0(X,L)$$ 
which are generically non-zero.
Consider the divisor 
\begin{equation}\label{divisor of lambda}
D'=\divisor(\pm s').
\end{equation} 
Then we have $D_s=2D'$ and
\begin{equation}\label{degree divisor xi not spectral curve}
\deg(D')=d_L.
\end{equation}
Let $M_1,M_2\subset V$ be defined by
$$M_1=\ker(\varphi-s' )\qquad\text{and}\qquad 
M_2=\ker(\varphi+s' ).$$ 
Since, for $i=1,2$, $M_i$ is locally-free and $V/M_i$ is torsion-free,
it follows that $M_i$ is in fact a line subbundle of $V$. Moreover, on
the complement of the divisor $D'$, we have a decomposition $V = M_1
\oplus M_2$ with respect to which $\varphi = \left(\begin{smallmatrix}
    s' & 0 \\
    0 & -s'
\end{smallmatrix}
\right)$.
Even though this does not extend over $D'$, we shall still refer to
the $M_i$ as \emph{eigenbundles} of $\varphi$. 

Let
\begin{equation}\label{divisor of wedge}
   {D}=\divisor(\epsilon),
\end{equation}
where 
\begin{equation}
\label{eq:epsilon}
\begin{aligned}
  \epsilon\colon M_1M_2&\longrightarrow\Lambda^2V \\
  x \otimes y &\longmapsto x\wedge y
\end{aligned}
\end{equation}
is induced by the inclusion $M_1\oplus M_2\hookrightarrow V$.
Then by Proposition~\ref{prop:M-plus-minus} we have
$M_1M_2\cong\Lambda(- {D})$ and hence for $i=1,2$
\begin{equation}
  \label{eq:deg-M_i}
  d/2-\deg( {D}) \leq \deg(M_i)\leq d/2,
\end{equation}
where the second inequality follows from semistability of
$(V,\varphi)$. 




Next we carry out a more careful analysis of $\varphi$ over $D'$. 
Recalling that $\Lambda =
\Lambda^2V$, we have the canonical extension
\begin{equation}\label{one more extension!}
  0\longrightarrow M_1\longrightarrow V\xra{v\mapsto v\wedge-} M_1^{-1}\Lambda
  \longrightarrow 0.
\end{equation}
Over a small open set $U$ in $X$, we can choose a splitting $V_{|U}\cong
M_{1|U}\oplus (M_1^{-1}\Lambda)_{|U}$. Using the definition of $M_1$
and the fact $\tr(\varphi)=0$ we see that, with respect to this
decomposition, $\varphi_{|U}$ is of the form
\begin{equation}\label{form of linegamma U}
\varphi_{|U}=\begin{pmatrix}
    s'  & q \\
     0 & -s' 
\end{pmatrix}
\end{equation}
for some $q$. We shall show that the restriction of $q$ to $D'$ (in
the scheme theoretic sense) is independent of the choice of splitting.  

\begin{proposition}
  \label{prop:q-varphi}
  There is a well defined section
  \begin{equation}\label{def thetagamma 2}
    q_\varphi\in H^0(D',M_1^2L\Lambda^{-1})
  \end{equation}
  given by restriction of $\varphi+s'\Id$ to $D'$.
\end{proposition}

\begin{proof}
  Restrict the sequence (\ref{one more extension!}) to $D'$.  Since
  $\varphi(t) = s't$ for any local section $t$ of $M_1$ and $s'$
  vanishes over $D'$, the restriction $\varphi_{|D'}$ factors through
  $(M_1^{-1}\Lambda)_{|D'}$. The restriction of $s'\Id\colon V \to VL$
  similarly factors and hence so does
\begin{math}
  (\varphi+s'\Id)_{|D'}\colon  V_{|D'} \to (VL)_{|D'},
\end{math}
giving a map
\begin{displaymath}
  \varphi_1\colon  (M_1^{-1}\Lambda)_{|D'} \to (VL)_{|D'}.
\end{displaymath}
Now the local form (\ref{form of linegamma U}) of $\varphi$ shows
that the composition $\varphi_1$ with the projection $(VL)_{|D'} \to
(M_1^{-1}\Lambda)_{|D'}$ vanishes. Hence we have a lift of $\varphi_1$
to a map
\begin{displaymath}
  q_{\varphi}\colon (M_1^{-1}\Lambda)_{|D'} \to (M_1L)_{|D'}
\end{displaymath}
as claimed.
\end{proof}

We have that $q_\varphi$ is a holomorphic section of
$M_1^2L\Lambda^{-1}$, over the subscheme $D'$ so, if
$q_\varphi\neq 0$ and $\divisor(q_\varphi)$ is the
corresponding divisor, then
\begin{equation}\label{divisor thetaphi >0}
\divisor(q_\varphi)\geq 0.
\end{equation}

\begin{lemma}\label{0<line D<tilde D}
  Let $q_\varphi$ be the section given in Proposition~\ref{prop:q-varphi},
  let $D'$ be the effective divisor defined in (\ref{divisor of
    lambda}) and let $ {D}$ be the effective divisor defined in
  (\ref{divisor of wedge}). Then the following statements hold.
\begin{enumerate}
\item If $q_\varphi\neq 0$, then $\divisor(q_\varphi)=D'- {D}$.
\item $0\leq  {D}\leq D'$.
\item $q_{\varphi}=0$ if and only if $ {D}=0$.
\end{enumerate}
\end{lemma}
\proof First note that (2) follows immediately from (1) and
(\ref{divisor thetaphi >0}).
In order to prove (1) and (3), we consider the local
form (\ref{form of linegamma U}) of $\varphi$ in a neighbourhood $U$
of a point $p \in X$.  Let $k_1(p)=\ord_p(s'
)\geq 0$ and $k_2(p)=\ord_p(q)\geq 0$ so that, if $z$ is a local
coordinate on $U$ centred at $p$, then 
\begin{displaymath}
s' (z)=z^{k_1(p)}f_s' (z)
\qquad\text{and}\qquad q(z)=z^{k_2(p)}f_q(z)
\end{displaymath}
with $\ord_0(f_s' )=\ord_0(f_q)=0$.

Let $(z_1,z_2)$ denote the coordinates on $\C\oplus\C$. We have chosen
a trivialization of $V$ over $U$ such that $M|_{U}$ is defined by the
equation $z_2=0$, i.e., $M|_{U}$ is the subspace of $\C\oplus\C$
generated by the vector $(1,0)$.  Also, the other
eigenbundle $M^{-1}\Lambda(- {D})|_{U}$ is defined by the
equation
$$2s'z_1+q z_2=0.$$
We have the following observations:
\begin{itemize}
 \item if $k_1(p)<k_2(p)$, then $M^{-1}\Lambda(- {D})|_{U}$ is the subspace of $\C\oplus\C$ generated by the vector $(0,1)$, thus $\epsilon:M_1M_2\to\Lambda^2V$ does not vanish in $p$;
\item if $k_1(p)=k_2(p)$, then $M^{-1}\Lambda(- {D})|_{U}$ is the subspace of $\C\oplus\C$ generated by the vector $(1,-2f_s' (0)f_q(0)^{-1})$, so again $\epsilon$ is not zero in $p$;
\item if $k_1(p)>k_2(p)$, then $M^{-1}\Lambda(- {D})|_{U}$ is the subspace of $\C\oplus\C$ generated by the vector $(1,0)$, and now $\epsilon$ vanishes in $p$ with $\ord_p(\epsilon)=k_1(p)-k_2(p)>0$.
\end{itemize}
By definition, $ {D}=\divisor(\epsilon)$, so we have just seen that 
$$ {D}(p)=\begin{cases}
k_1(p)-k_2(p) &\text{ if }k_1(p)>k_2(p)\\
\hspace{0,1cm}0 &\text{ otherwise}.
\end{cases}$$
Furthermore, $D'=\divisor(s' )$, i.e., $D'(p)=\ord_p(s' )=k_1(p)$. Hence,
\begin{equation}\label{line D=tilde D-k2}
 {D}(p)=\begin{cases}
D'(p)-k_2(p) &\text{ if }D'(p)>k_2(p)\\
\hspace{0,1cm}0 &\text{ otherwise}.
\end{cases}
\end{equation}
Now notice that at each point $p\in\Supp(D')$ the section
$q_\varphi$ is given exactly by the value of $q$ at $p$
(cf.\ (\ref{form of linegamma U})),
i.e., 
$$\divisor(q_\varphi)(p)=\ord_p(q)=k_2(p).$$ 
Hence, (\ref{line D=tilde D-k2}) proves (1).

Finally (3) also follows from (\ref{line D=tilde D-k2}).
\endproof

\begin{remark}
  In the situation of irreducible spectral curve $X_s$, the
  eigenbundles are only globally well defined for the pull back of
  $(V,\varphi)$ to $X_s$.  However, one may still define a divisor
  $ {D}$ on $X$ by using the locally defined eigenbundles, and
  it turns out that $ {D}=D'$, in contrast to the present
  situation. The basic reason for this is that on the spectral curve
  the eigenbundles are interchanged by the involution on $X_s$, thus
  tying them more tightly to each other.
\end{remark}

Finally we can finish the proof of Theorem~\ref{prop:fibre when
  spectral curve is reducible}. 
We start by proving (1).  Let
$(V,\varphi)\in\cH^{-1}(s)$. In the preceding we have constructed a
divisor $ {D}$ satisfying $0 \leq  {D} \leq D'$ (by
Lemma~\ref{0<line D<tilde D}) and a $\varphi$-invariant line subbundle
$M_1\subset V$ such that $m = \deg(M_1)$ satisfies the bound
(\ref{eq:bound-m}) (by (\ref{eq:deg-M_i})). We also constructed
$q_{\varphi} \in H^0(D',M_1^2L\Lambda^{-1})$, which by
Lemma~\ref{0<line D<tilde D} satisfies $q_\varphi\in
E( {D},M)$.  Thus we have associated to $(V,\varphi) \in
\mathcal{H}^{-1}(s)$ an element
$$
\zeta(V,\varphi) = (q_{\varphi},M_1) \in \mathcal{E}( {D},m).
$$ 
Analysing the construction of the morphism $p$ given in
Section~\ref{sec:proof-lemma-morphism}, it is easy to see that
$$p(\zeta(V,\varphi))=(V,\varphi).$$
This proves the surjectivity of $p$.

In order to prove the remaining claims of Theorem~\ref{prop:fibre when
  spectral curve is reducible}, let $M_i$ be the eigenbundles of a pair
$(V,\varphi)$ for $i=1,2$. We carried out the preceding construction
of $\zeta(V,\varphi)$ using the eigenbundle $M_1$, but we could
equally well have carried it out using the eigenbundle $M_2$, which
has degree $\deg(M_2) = d - \deg(  D) - \deg(M_1)$. This
proves (2) of Theorem~\ref{prop:fibre when
  spectral curve is reducible}. Finally the remaining claims of the theorem
follow because it is easy to see that the choice of the eigenbundle
$M_1$ is the only ambiguity in the construction. Thus $p\colon
\mathcal{E}(  D,m) \to \mathcal{H}^{-1}(s)$ can only fail to
be injective if $m = d - \deg(  D) - m \iff m =
(d-\deg(D ))/2$ and the eigenbundles are non isomorphic, in which
case $p$ is two to one.

\subsection{Connectedness and dimension of the fibre}
\label{sec:connectedness-dimension-fibre-redicble-case}
For a given $m$, let $d_m$ be the smallest integer greater than or
equal to $d/2-m$, i.e.,
\begin{equation}
  \label{eq:d_m}
  d_m = \lceil d/2-m\rceil
\end{equation}
and define
\begin{displaymath}
  \mathcal{E}(m) = \bigcup_{\deg(D )\geq d_m} \mathcal{E}(D ,m).
\end{displaymath}
Then Theorem~\ref{prop:fibre when spectral curve is reducible} shows
that
\begin{equation}
  \label{eq:decomposition-m}
  \mathcal{H}^{-1}(s) = \bigcup_{d/2-d_L \leq m \leq d/2} p(\mathcal{E}(m))
\end{equation}

\begin{lemma}
  \label{lem:Em-connected}
  The subspace $p(\mathcal{E}(m)) \subset \mathcal{H}^{-1}(s)$ is
  connected for any $m$ satisfying $d/2-d_L \leq m \leq d/2$.
\end{lemma}

\begin{proof}
  It suffices to see that $\mathcal{E}(m)$ is connected. But from (1)
  of Proposition~\ref{prop:properties-EDM} we have
  \begin{displaymath}
    \bigcup_{\deg(D )\geq d_m}E(D,M) = H^0(D',M^2L\Lambda^{-1}).
  \end{displaymath}
  Hence $\mathcal{E}(m)$ is simply the natural vector bundle over
  $\Jac^m(X)$ with fibres $H^0(D',M^2L\Lambda^{-1})$.
\end{proof}

We now prove that the fibre of $\cH$ over a point such that the
spectral curve is reducible, is connected.

\begin{theorem}\label{thm: fibre spectral curve reducible is connected}
  Let $s\in H^0(X,L^2)\setminus\{0\}$ such that $X_s$ is
  reducible. Then the fibre of $\cH:\cM_L^\Lambda\to H^0(X,L^2)$ over
  $s$ is connected.
\end{theorem}
\proof
For any $m$, take a divisor $D $ of degree $d_m$ (defined in
(\ref{eq:d_m})). Then
\begin{displaymath}
  d - \deg(D ) - m = [d/2]
\end{displaymath}
and (2) of Theorem~\ref{prop:fibre when spectral curve is reducible}
shows that
\begin{displaymath}
  p(\mathcal{E}(m)) \cap p(\mathcal{E}([d/2])) \neq \emptyset.
\end{displaymath}
The conclusion is now immediate from (\ref{eq:decomposition-m}) and
Lemma~\ref{lem:Em-connected}.
\endproof

Now we compute the dimension of every stratum of the fibre and, since
it is connected, of the fibre itself.

\begin{proposition}\label{prop:dimension of strata}
If $s\in H^0(X,L^2)$ is such that $D_s=2 D'$ and $L=\mathcal O( D')$, then:
\begin{enumerate}
 \item $\dim\mathcal F(0,d/2)=g$;
 \item for any effective divisor $  D\leq D'$ and $m$
   satisfying (\ref{eq:bound-m}), $\dim\mathcal F( 
   D,m)=\deg(D)+g-1$;
\item $\dim\cH^{-1}(s)=d_L+g-1$.
\end{enumerate}
\end{proposition}
\proof Immediate from the definition of $\mathcal{F}(D,m)$ and
Theorem~\ref{prop:fibre when spectral curve is reducible}.
\endproof

Hence, every stratum of $\cH^{-1}(s)$ has dimension less or equal than
$\deg(D')+g-1=d_L+g-1$ and this upper bound is reached only by
$\mathcal F(D',m)$ for any integer $m$ such that $d/2-d_L\leq m\leq d/2$.

\begin{remark}\label{rem:fibre for higher degree line bundle}
Recall that $L\cong\mathcal{O}(D')$. Let now $\tilde D$ be a divisor such that $\tilde D>D'$, and consider $\tilde L\cong\mathcal{O}(\tilde D)$. Denote by $\cM_{\tilde L}^\Lambda$ the moduli space of $\tilde L$-twisted Higgs pairs and let $\tilde\cH:\cM_{\tilde L}^\Lambda\to H^0(X,\tilde L^2)$ be the corresponding Hitchin map. Let $\tilde E(D,m)$ be the space defined in the same way as $E(D,m)$ in (\ref{E(D,M)}), but now using $\tilde D$ instead of $D'$. By Remark \ref{restriction to a divisor 1} we see that, for an effective divisor $D$ such that $D\leq D'$, there is a natural inclusion $$E(D,M)\hookrightarrow\tilde E(\tilde D-D'+D,M),$$ which induces an inclusion $$\mathcal{E}(D,m)\hookrightarrow\tilde{\mathcal{E}}(\tilde D-D'+D,m),$$ where $\tilde{\mathcal{E}}(\tilde D-D'+D,M)$ is the bundle over $\Jac^m(X)$ whose fibre over $M$ is $\tilde E(\tilde D-D'+D,M)$, as in Definition \ref{bundle E(D,m)}.
Now, we have a map $p:\mathcal{E}(D,m)\to\cH^{-1}(s)$, depending on a choice of a square root $s'$ of $s$ (see Remark \ref{dependence of p on square root of s}), satisfying the conditions of Theorem \ref{prop:fibre when spectral curve is reducible}. Given $\tilde s'\in H^0(X,\tilde L)$, we can apply Theorem \ref{prop:fibre when spectral curve is reducible} to $\tilde\cH^{-1}(\tilde s'^2)$, to get a map $\tilde p:\tilde{\mathcal{E}}(\tilde D-D'+D,m)\to\tilde\cH^{-1}(\tilde s'^2)$, depending on the choice of $\tilde s'$, and satisfying all the conditions stated in that theorem. This, together with the above inclusions, induces an inclusion $\cH^{-1}(s)\hookrightarrow\tilde\cH^{-1}(\tilde s'^2)$.

In other words, the fibre of $\cH:\cM_L^\Lambda\to H^0(X,L^2)$ over $s\in H^0(X,L^2)$ such that $X_s$ is reducible, can be identified as a subvariety of the fibre of $\tilde\cH:\cM_{\tilde L}^\Lambda\to H^0(X,\tilde L^2)$, over an $\tilde s\in H^0(X,\tilde L^2)$ such that $X_{\tilde s}$ is reducible and $\tilde L$ is a line bundle of the form $\tilde L\cong\mathcal O(\tilde D)$ for some $\tilde D>D'$.
\end{remark}

\section{Main Theorem}

We finish by putting everything together to obtain the following main
result.

\begin{theorem}\label{thm:main}
  Let $L\to X$ be a line bundle with $\deg(L)>0$ and let
  $\cH:\cM^\Lambda_L\to H^0(X,L^2)$ be the Hitchin map. For any
  section $s\in H^0(X,L^2)$, the fibre $\cH^{-1}(s)$ is
  connected. Moreover, if $s \neq 0$, the dimension of the fibre is
  $\dim(\cH^{-1}(s))=d_L+g-1$.
\end{theorem} 

\proof 
Assume first that $s \neq 0$.
The connectedness of $\cH^{-1}(s)$ is immediate from Theorems~\ref{thm: fibre connected}, and \ref{thm: fibre spectral curve
  reducible is connected}. The dimension formula follows from
(\ref{dimension of generic fibre}), and Propositions
\ref{prop:dimension singular fibre irreducible spectral curve} and
\ref{prop:dimension of strata}.

It remains to prove that $\cH^{-1}(0)$ is connected\footnote{This argument was outlined to
  us by the referee.}.  Assume that this
is not the case. Then there exist non-empty open sets $U$ and $V$ in
$\cM^\Lambda_L$ such that $\cH^{-1}(0) \subseteq U \cup V$. Since
proper maps are closed, there is an open ball around zero, $W \subset
H^0(X,L^2)$, such that $\cH^{-1}(W) \subseteq U \cup V$. Hence we may
write $\cH^{-1}(W) = \tilde{U} \cup \tilde{V}$ for non-empty disjoint
open sets $\tilde{U}$ and $\tilde{V}$. The connectedness of the fibres
$\cH^{-1}(s)$ for $s \neq 0 $ implies, again using that $\cH$ is
closed, that $\cH^{-1}(W \setminus \{0\})$ is connected. Hence either
$\tilde{U}$ or $\tilde{V}$ is contained in the fibre over zero and is
therefore a connected component of the moduli space
$\cM^\Lambda_L$. Since the moduli space is connected
(Proposition~\ref{prop:moduli-connected}) we have reached a
contradiction.
\endproof

\begin{remark}
\label{rem:deg-L-zero}
  Throughout this paper we have always considered a line bundle $L$ of
  positive degree. Let us say a few words about the situation for $L$
  with $\deg(L)=0$.

  Suppose that $\deg(L)=0$ and that $s\in H^0(X,L^2)$ is
  non-zero. Hence, $L^2\cong\mathcal O$ and $s\in\C^*$.  Suppose that
  $s$ does not admit a square root. Then the spectral curve
  $\pi:X_s\to X$ is a non-trivial unramified double cover. In this
  case, we are in the smooth generic situation, and
  $H^{-1}(s)\cong\Prym_\pi(X_s)$. However, as shown in
  \cite{mumford:1971}, $\Prym_\pi(X_s)$ has two connected components,
  thus the fibre of $\cH$ over $s$ is disconnected.

  On the other hand, if $s=s'^2$ for some $s'\in\C^*$, then
  $L\cong\mathcal O$, and $X_s$ is a disconnected double cover of
  $X$. If $(V,\varphi)\in\cH^{-1}(s)$, then $V\cong M\oplus\Lambda M^{-1}$ and
\begin{displaymath}
\varphi=\pm\left(\begin{array}{cc}\sqrt{-1}s' & 0 \\
0 & -\sqrt{-1}s'\end{array}\right),\quad \text{for some $M\in\Jac^{d/2}(X_s)$}.
\end{displaymath} 
It follows that there is a double cover $\Jac^{d/2}(X_s) \coprod
\Jac^{d/2}(X_s) \to\cH^{-1}(s)$, ramified over pairs $(V,\varphi)$
such that $V\cong M\oplus\Lambda M^{-1}$, with $M^2\cong\Lambda$. In other
words, $\cH^{-1}(s)$ is the disjoint union of two copies of
$\Jac^{d/2}(X_s)$ glued together along the subvariety defined by
$M^2\cong\Lambda$. Hence $\cH^{-1}(s)$ is connected and has dimension
$g$. Finally, notice that the image of $p:\mathcal
F(0,d/2)\to\cH^{-1}(s)$ in (4) of Theorem \ref{prop:fibre when
  spectral curve is reducible} (in the case of $\deg(L)=d_L>0$) can be identified with the image of the above ramified double cover
$\Jac^{d/2}(X_s)\to\cH^{-1}(s)$ for $L\cong\mathcal O$. This is an
example of the phenomenon described in Remark \ref{rem:fibre for
  higher degree line bundle}.
\end{remark}



\end{document}